\documentclass[11pt]{amsart}
\usepackage{graphicx} %
\usepackage{amsfonts, amsmath, amsthm, amssymb}
\usepackage{mathtools}
\usepackage[dvipsnames]{xcolor}
\usepackage[
    pdftex=true,
    colorlinks=true,
    linkcolor=RoyalPurple,   %
    citecolor=Blue,     %
    urlcolor=Violet     %
]{hyperref}
\usepackage[margin=1.25in]{geometry}
\usepackage{float, comment}
\usepackage{algorithm}
\usepackage{algorithmic}

\usepackage{cleveref}
\usepackage[most]{tcolorbox}
\usepackage[utf8]{inputenc}
\usepackage[T1]{fontenc}
\usepackage{empheq}
\usepackage{amsfonts}
\usepackage{tikz-cd}
\usepackage{caption}
\usepackage{subcaption}
\usepackage{url}

\newtcbtheorem{defn}{Definition}%
{	colframe=blue!30!gray, theorem name, fonttitle=\bfseries,
	colback=blue!5, top=1mm,bottom=1mm, boxrule=0.5pt}{Example}

\newtcbtheorem[use counter=theorem, number within=section]{THM}{Theorem}%
{	colframe=blue!30!gray, fonttitle=\bfseries,
	colback=blue!5,  fontupper=\itshape, top=1mm,bottom=1mm, boxrule=0.5pt}{Theorem}

\newtcbtheorem[use counter=theorem, number within=section]{EX}{Example}%
{	colframe=blue!30!gray, fonttitle=\bfseries,
	colback=blue!5,  fontupper=\itshape, top=1mm,bottom=1mm, boxrule=0.5pt}{Theorem}

\usepackage{varwidth}
\newtcbox{\mymath}[1][]{%
	nobeforeafter, math upper, tcbox raise base,
	enhanced, colframe=blue!30!black,
	colback=blue!5, boxrule=0.5pt, top=1mm,bottom=1mm,
	#1}
 \newtcbox{\mybox}{colback=blue!5,
	colframe=blue!30!black, center, enhanced, varwidth upper}

\usepackage{biblatex}
\theoremstyle{plain}
\newtheorem{theorem}{Theorem}[section]
\newtheorem{lemma}[theorem]{Lemma}
\newtheorem{proposition}[theorem]{Proposition}
\newtheorem{corollary}[theorem]{Corollary}

\theoremstyle{definition}
\newtheorem{definition}[theorem]{Definition}
\newtheorem{assumption}[theorem]{Assumption}
\newtheorem{example}{Illustration}
\newtheorem{fact}{Fact}

\Crefname{example}{Illustration}{Illustrations} %
\crefname{example}{Illustration}{Illustrations} %
\def\T{{\top}}

\def\Tr{{\operatorname{Trace}}}
\def\Span{{\operatorname{span}}}

\def\argmin{{\operatorname{argmin}}}

\def\1{{\mathbbm{1}}}
\def\E{{\mathbb{E}}}

\def\phi{{\varphi}}

\def\R{{\mathbb R}}
\def\Card{{\mathrm{Card}}}
\def\gcd{{\mathrm{gcd}}}
\def\F{{F}}

\def\Pol{{\rm Pol}}
\def\R{{\mathbb{R}}}
\def\fS{{\mathfrak{S}}}
\def\a{{\alpha}}
\def\N{{\mathbb{N}}}
\def\rank{\operatorname{rank}}
\def\sgn{\operatorname {sgn}}
\def\a{\alpha}
\def\b{\beta}

\def\d{\delta}
\def\e{\epsilon}

\def\l{\lambda}

\def\s{\sigma}
\def\t{\tau}
\def\Z{{\mathbb{Z}}}
\def\cH{{\mathcal{H}}}
\def\cM{{\mathcal{M}}}
\def\cV{{\mathcal{V}}}
\def\cG{{\mathcal{G}}}
\def\GL{{\rm GL}}
\def\per{{\rm per}}
\def\proj{{\rm proj}}
\def\id{{\rm id}}
\DeclarePairedDelimiter{\dotp}{\langle}{\rangle}
\def\PP{{\mathbb{P}}}
\newcommand{\EE}{\mathbb{E}}
\newcommand{\cF}{\mathcal{F}}
\def\cE{{\mathcal{E}}}

\floatname{algorithm}{Procedure}

\usepackage[draft]{fixme}
\fxsetup{theme=color, mode=multiuser, noinline}
\FXRegisterAuthor{rt}{RT}{\color{red}Rekha}
\FXRegisterAuthor{jk}{JK}{\color{blue}Jack}
\FXRegisterAuthor{dd}{DD}{\color{teal}Dima}
\FXRegisterAuthor{md}{MD}{\color{magenta}Mateo}

\title{Invariant kernels: rank stabilization and generalization across dimensions}

\author{Mateo D\'{i}az}
\address{Department of Applied Mathematics and Statistics,
Mathematical Institute for Data Science, 
Johns Hopkins University, Baltimore, MD; \href{https://mateodd25.github.io/}{\texttt{mateodd25.github.io}}.}
\author{Dmitriy Drusvyatskiy}
\address{Department of Mathematics, U. Washington,
	Seattle, WA; \href{https://sites.google.com/uw.edu/ddrusv/home}{\texttt{www.math.washington.edu/$\sim$ddrusv}}. Research supported by NSF DMS-2306322, NSF CCF 1740551, AFOSR FA9550-24-1-0092 awards.}
\author{
Jack Kendrick}
\address{Department of Mathematics, U. Washington,
	Seattle, WA; \href{mailto:jackgk@uw.edu}{\texttt{jackgk@uw.edu}}.} 

\author{Rekha R. Thomas}
\address{Department of Mathematics, U. Washington,
	Seattle, WA;	\href{https://sites.math.washington.edu/~thomas/}{\texttt{sites.math.washington.edu/$\sim$thomas/}}.}

\begin{document}

\begin{abstract}
Symmetry arises often when learning from high dimensional data.  For example, data sets consisting of point clouds, graphs, and unordered sets appear routinely in contemporary applications, and exhibit rich underlying symmetries. 
Understanding the benefits of symmetry on the statistical and numerical efficiency of learning algorithms is an active area of research. In this work, we show that symmetry has a pronounced impact on the {\em rank of kernel matrices}. Specifically, we compute the rank of a polynomial kernel of fixed degree that is invariant 
under various groups acting independently on its two arguments. In concrete circumstances, including the three aforementioned examples, symmetry dramatically decreases the rank making it {\em independent of the data dimension}. In such settings, we show that a simple regression procedure is minimax optimal for estimating an invariant polynomial from finitely many samples drawn across different dimensions.
We complete the paper with numerical experiments that illustrate our findings.
\end{abstract}

\maketitle

\section{Introduction}
Contemporary applications in the domain sciences and engineering often involve data sets that are invariant under some problem-specific group action. For example, graph properties are invariant under relabeling of vertices, point clouds are invariant under orthogonal transformations, unordered sets are invariant under permutations, and so forth. 
Estimation and inference with such data sets can benefit from taking invariance into account. %
In practice, one often explicitly hard-codes invariances into the parametric description/architecture of the function class. High-impact examples of invariant architectures include DeepSets \cite{zaheer2017deep} for sets, Convolutional Neural Networks for images \cite{krizhevsky2012imagenet}, PointNet \cite{qi2017pointnet,qi2017pointnet++} and tensor field Neural Networks \cite{thomas2018tensor} for point clouds, and Graph Neural Networks \cite{scarselli2008graph} for graphs. A key question in this line of work therefore is to understand how group invariance of the data benefits both computation and sample efficiency of learning. Seeking to answer this question, a number of recent papers \cite{tahmasebi2023exact,bietti2021sample,mei2021learning} have focused on kernel methods, showing that group invariance improves the standard sample complexity bounds by effectively scaling the dimension of the problem's data by the size of the group. Our work is largely complementary to this effort. We show that group invariance can significantly decrease the rank of polynomial kernels, which can be beneficial both for computation and inference.

Setting the stage, a {\em kernel} $K\colon V\times V\to\R$ {\em on a set $V$} is a symmetric bivariate function. Throughout, we will assume that $V$ is a finite-dimensional vector space; e.g. $V=\mathbb{R}^d$ or $V=\mathbb{R}^{d\times d}$.
The primary use of kernels in applications is to measure pairwise similarity between points 
$x_1,\ldots, x_n\in V$ via the matrix  $\mathbf{K}=[K(x_i,x_j)]_{i,j=1,\ldots,n}$.
A kernel $K$ admits a {\em rank-$r$ factorization} if it can be written as 
\begin{equation}\label{eqn:basic_kern_intro}
K(x,x')=\langle \varphi(x),\psi(x')\rangle\qquad \forall x,x'\in V,
\end{equation}
for some maps 
$\varphi,\psi\colon V\to\R^r$.
The {\em rank} of $K$, denoted $\rank K$, is the minimal such $r$.

 Low-rank kernels are  desirable both for computational and statistical reasons. Indeed, having a low rank-factorization ${\bf K}=AB^\top$ enables cheap matrix vector products and therefore faster linear algebraic computations. More generally, state-of-the-art kernel system solvers are based on cheaply forming low-rank approximations of kernel matrices and using them as preconditioners; e.g. RPCholesky \cite{chen2022randomly, diaz2023robust}, random Fourier Features \cite{avron2017faster,cutajar2016preconditioning}, EigenPro \cite{ma2017diving,ma2019kernel}, and Falcon \cite{falkonlibrary2020,falkonhopt2022}. From a statistical perspective, the classical generalization bound for the ordinary least squares estimator for linear regression $y=\langle \varphi(x),\beta^* \rangle+\varepsilon$ scales as $O(\sigma^2\rank(K)/n)$ (e.g. \cite[Exercise 13.2]{wainwright2019high}); therefore $\rank K$ serves as the intrinsic dimension of the problem.

In this work, we focus on polynomial kernels of a fixed degree $m$. By this we mean that $K(x,y)$ is a polynomial of degree $m$ in each of the arguments $x$ and $y$. We emphasize that we will regard $m$ as being a fixed constant throughout. The reader should think of $m$ as being relatively small while the dimension of the data $d$ and the number of data points $n$ are large.
A baseline example to keep in mind is the standard polynomial kernel $K(x,y)=(1+x^{\top} y)^m$ on $\R^d$ which can be realized by the feature maps $\varphi=\psi$ that evaluate all monomials $x^\a$ with $|\a|\leq m.$ %
In particular, $K$ has rank  $\sum_{j=0}^m{d+j-1\choose j}\asymp d^m$ (ignoring constants in $m$). %

In this work, we investigate the impact of symmetry on the kernel rank. Namely, given a group $G$ acting on $V$, we say that $K$ is {\em $G$-invariant} if the following equality holds:%
\footnote{$G-$invariance should be contrasted with the much weaker property of $G-$translation invariance, which stipulates that \eqref{eqn:g_invar} holds only with $g=g'.$ In contrast to $G-$invariant kernels, the rank of $G-$translation invariant kernels is usually high.}
\begin{equation}\label{eqn:g_invar}
  K(g x,g' x')=K(x,x')\qquad \forall x,x'\in V, ~g,g'\in G.  
\end{equation}
In particular, if $G$ is a finite group then any kernel induces a $G$-invariant kernel by averaging; see Section~\ref{sec:gi_sym} for details.
We investigate the rank of $G$-invariant kernels of degree $m$ on a $d$-dimensional vector space $V$, which should always be compared to the rank  $d^m$ of the %
standard polynomial kernel. Our contributions are as follows:
\mybox{\it
We will see that in a number of important settings the rank of $G$-invariant kernels is independent of the ambient dimension $d$ and derive consequences for sample efficiency of learning $G$-invariant polynomials when data is sampled across different dimensions. 
}

\subsection{Summary of results}
Much of our computations stem from the following elementary observation. The rank of any $G$-invariant degree $m$ kernel $K$ on a vector space $V$ can be upper-bounded as
\begin{empheq}[box=\mymath]{equation}
\rank K\leq \dim(\R[V]_m^G)\label{eqn:rank_est_basis0}
\end{empheq}
where $\R[V]_m$ denotes the vector space of polynomials on $V$ of degree at most $m$ and $\R[V]_m^G$ denotes the subspace of polynomials $\R[V]_m$ that are $G$-invariant.
Thus bounding the rank of $K$ is reduced to estimating the dimension of the invariant space $\R[V]_m^G$---a well-studied object, e.g., in combinatorics \cite{stanley1979invariants} and representation theory \cite{SerreBook}. 

We next describe some concrete examples where group invariance yields a rank reduction. The explicit bounds rely on so-called integer partitions. A {\em partition of an integer $j$}, denoted $\a\vdash j$, is a sequence $\a_1\geq \a_2\geq \ldots$ of strictly positive integers (called {\em parts}) summing to $j$. The symbol $p(j)$ will denote the number of integer partitions of $j$. %

\paragraph{\bf Permutation Invariant Functions}
As our very first example, we focus on functions on $V=\R^d$ that are invariant under coordinate permutations by the action of the symmetric group $\fS_d$. Common examples are coordinate sums, products, and functions of symmetric polynomials. It follows from well-known results---see, for example, \cite{fulton-harris}---that the dimension $\dim(\R[V]_m^{\fS_d}$ of the space of $\fS_d$-invariant polynomials is given by a quantity that only depends on the degree $m$ of the polynomials and not on the dimension $d$ of $V$. We record this observation in the following example, and provide a short argument in the text.

\begin{EX}[label={thm:Permutations}]{(Permutations)}{} 
    In the setting $d\geq m$, the quantity $\dim(\R[V]_m^{\fS_d})$ is independent of $\dim V$ and satisfies:
     \begin{equation}\label{eqn:rank_permut}
     \dim(\R[V]_m^{\fS_d}) = \sum_{j=0}^m p(j).
      \end{equation}
      \end{EX}

We stress that the bound \eqref{eqn:rank_permut} depends only on the degree $m$ and is at most $\exp(\pi\sqrt{2m/3})$ by the Hardy-Ramanujan asymptotic formula \cite{hardy-ramanujan}. Concretely, for $m=1,\ldots, 10$, we may compute $\sum_{j=0}^m p(j)$ directly yielding the bounds $2,4,7,12,19,30,45,67,97,139$, respectively.

\paragraph{\bf Set-permutation Invariant Functions.}
The next example considers invariances that are typically exhibited when each data point is itself a set of vectors. Such problems appear routinely in practice; see for example the influential work \cite{zaheer2017deep}. Specifically, let $S = \{x_1, \ldots, x_d\}$ be a set of $d$ points $x_i\in \R^k$. In applications, $x_i$ may denote the rows of a table $S$; as such, the number of rows $d$ is large, while each row $x_i$ lies in a low-dimensional space $\R^k$.
Clearly, $S$ can be represented by a matrix $X\in \R^{d\times k}$ having $x_i$ as its rows. For any permutation matrix $P\in \R^{d\times d}$, the matrices $PX$ and $X$ have the same rows in a different order and therefore, the two matrices represent the same set. Thus we will be interested in kernels on $V=\R^{d\times k}$ that are invariant under the action of $\fS_d$ by permuting rows. Existing results, such as those in \cite{fleischmann-invariantTheory, smith-InvariantTheory, weyl1946classical}, imply that $\dim(\R[V]_m^{\fS_d})$ is bounded by a quantity independent of the dimension. We provide a simple combinatorial argument to explicitly calculate the value of this dimension, summarized in the following example. %

\begin{EX}[label={thm:Set-permutations}]{(Set-permutations)}{} \label{thm:set-permutations}
When $d\geq m$, the quantity $\dim(\R[V]_m^{\fS_d})$ is independent of $\dim V$ and satisfies: 
    \begin{align*}
        \dim(\R[V]_m^{\fS_d})  = \sum_{j=0}^m\sum_{\a\vdash j}\prod_{i=1}^j{{i+k-1\choose i} + \mu_i(\a)-1\choose \mu_i(\a)},
    \end{align*}
    where $\mu(\a)$ is the vector of the multiplicities of parts of $\a.$
    \end{EX}
 Although the bound might look cumbersome, it can be easily computed for small values of $m$ and $k$, and most importantly, this bound is  independent of $d$. Using the Hardy-Ramanujan asymptotic formula %
 , the  bound in Theorem \ref{thm:Set-permutations} grows at most as $\exp(\pi\sqrt{2m/3})\cdot(k-1)^{mk}.$ %

\paragraph{\bf Functions on Graphs.}
Many problems arising in applications involve data in the form of weighted graphs. Recall that all weighted graphs on $d$ vertices can be uniquely represented by its $d\times d$ adjacency matrix which is 
invariant up to simultaneous permutation of rows and columns since the labeling of a graph's vertices is arbitrary. It follows that any function of graphs that takes adjacency matrices as input should be invariant under the action of $\fS_d$ by conjugation. Let $V$ consist of all symmetric $d\times d$ matrices. Clearly, all adjacency matrices of simple graphs on $d$ vertices are contained in $V$. Existing results, such as those contained in the OEIS sequence \cite{oeisA007717}, imply that the dimension $\dim(\R[V]_m^{\fS_d})$ is bounded by a quantity independent of $d.$ %
We prove an alternative upper bound for $\dim(\R[V]_m^{\fS_d})$ and show that the dimension $\dim(\R[V]_m^{\fS_d})$ stabilizes for all $d\geq 2m.$ %
\begin{EX}[label={theorem: graphs rank}]{(Graphs)}{}
In the setting $d\geq 2m$, the quantity $\dim(\R[V]_m^{\fS_d})$ is independent of $\dim V$ and satisfies:
\begin{align*}
    \dim(\R[V]_m^{\fS_d})\leq \sum_{j=0}^{m}\sum_{\a\vdash 2j}q_\a,
\end{align*}
where $q_{\a}$ is the number of non-negative integer symmetric matrices %
with row sums equal to $\a$.
\end{EX}
We are not aware of any procedure to compute $q_{\a}$ exactly. It can be upper bounded, however, by relaxing the symmetry assumption and stipulating that the row sums and columns sums are $\a$ and using the recursive procedure in \cite{miller-harrison}. %
A straightforward upper bound obtained by relaxing the symmetry assumption is $2m\cdot\exp(\pi\sqrt{m/12})\cdot(m-1)^{2mk}(2m)^{m(m-1)}$.%

\paragraph{\bf Functions on Point Clouds}
Problems in fields such as computer vision and robotics often involve data in the form of \textit{point clouds}, i.e. sets of vectors $S=\{x_1,\ldots, x_d\}\subset \R^k$. Clearly, we may store a point cloud as a matrix $X\in \R^{d\times k}$ having $x_i$ as its rows. 
Point clouds are invariant under the relabeling
of points in the cloud, as well as simultaneous orthogonal transformations. Let $G$ be the group of these symmetries acting on matrices $X\in \R^{d\times k}$.  The following result gives a dimension-independent upper bound on $\dim(\R[V]_m^G)$.

\begin{EX}[label={theorem: point cloud rank}]{(Point clouds)}{}
In the regime $d\geq 2m$, the quantity $\dim(\R[V]_m^{G})$ is independent of $\dim V$ and satisfies: 
\begin{align*}
    \dim(\R[V]_m^{G})\leq \sum_{j=0}^{m/2}\sum_{\a\vdash 2j}q_\a,
\end{align*}
where $q_{\a}$ is the number of non-negative integer symmetric matrices with row sum equal to $\a$. %
\end{EX}

 \paragraph{\textbf{A Monte Carlo rank estimator}} Computing the dimension of invariants, as in the aforementioned examples, typically reduces to tedious combinatorial arguments.
We now describe an alternative approach, which allows one to estimate the dimension of the invariants simply by sampling elements of the group.
The statement of the theorem requires some further notation, which we outline now; see Section~\ref{sec:notation_part_comp} for details. When $G$ is a finite group acting on $V$, the symbol $\E_g[\cdot]$ will mean expectation with respect to the uniform measure on $G$. 
For any integer $k$, the symbol {$C_d(k)$ denotes the set of all weak compositions of $k$ into $d$ parts and $\eta(\a)$ denotes the product of the factorials of the multiplicities of non-negative entries in $\a$. 
The symbols $\lambda_i(g)$ denote the eigenvalues of $g\in G$ represented as a matrix, while $g[\alpha]$ and ${\rm per} \, g[\alpha]$ denote the submitrix of $g$  indexed by $\alpha$  and its permanent.
\begin{THM}[label={thm:calculate_dim_intro}]{(Calculating the rank)}{}
     Let $G$ be a finite group acting linearly on a $d$-dimensional vector space $V$. Then,
\begin{equation}\label{eqn:sophisticated}
\dim(\R[V]_m^G)=\E_{g} \left[\sum_{j=0}^m\sum_{\a\in {C_d(j)}}\prod_{i=1}^d \lambda_{i}(g)^{\a_i}\right],\quad\dim(\R[V]_m^G)=\E_{g} \left[\sum_{j=0}^m \sum_{\a\in {C_d(j)}} \frac{{\rm per}\,  
 g[\a]}{\eta(\a)}\right].
\end{equation}
\end{THM}
Although the first equation in \eqref{eqn:sophisticated} looks complicated, it is not very difficult to compute the terms inside the expectation since the degree $m$ is assumed to be small. Consequently, even when the group $G$ is very large one may estimate $\dim(\R[V]_m^G)$ by drawing iid samples from $G$ and replacing $\E_g$ by an imperical average.
The downside of the expression is that one needs to compute eigenvalues $\lambda_{i}(g)$ of a $d\times d$ matrix. The second equality in \eqref{eqn:sophisticated} provides an alternative, wherein one instead needs to only compute the permanent of small submatrices of $g$ whose columns and rows are indexed by the weak composition $\a$. Since $\a$ is a composition of at most $m,$ each submatrix is at most $m\times m.$ Again, even when the group $G$ is very large one may estimate $\dim(\R[V]_m^G)$ by using an empirical average.

\paragraph{\bf Connection to representation stability} The stabilization of the dimension of fixed-degree invariant polynomials is not merely coincidental, but rather a consequence of a broader phenomenon known as \emph{representation stability} \cite{church2013representation, church2014representation}. To elucidate this connection, consider a sequence of nested groups $\{G_d\}_{d\in\N}$ (e.g., $G_d = \fS_d$) acting on a sequence of vector spaces $\{W_d\}_{d\in\N}$ of increasing dimension that embed isometrically into each other $W_d \hookrightarrow W_{d+1}$ (e.g., $W_d = \R[V_d]_m$). Representation stability asserts that if these sequences satisfy certain algebraic relationships, then the subspaces of invariants $W_d^{G_d} = \{ w \in W_d \mid g \cdot w = w \,\,\forall g \in G_{d}\}$ become isomorphic to each other.
The following result, which can also be deduced from the more general categorical framework in \cite{levin2023free}, is presented here via direct and elementary arguments.
\begin{THM}[label={thm:free-basis}]{(Polynomial basis across dimensions)}{}
    Let $\{G_d\}_{d\in\N}$ and $\{V_d\}_{d\in\N}$ be any of the sequences of groups and vector spaces considered in Examples~\ref{thm:Permutations}-\ref{theorem: point cloud rank} and define $W_d:= \R[V_d]_m.$ Then, for each example, there exist $r \in \N$ and invariant polynomials $\left\{f_1^{(d)}, \dots, f_r^{(d)} \in W_d^{G_d}\right\}_{d\in\N}$ such that 
    \begin{enumerate}
      \item (\textbf{Basis}) For all large $d$, the polynomials $f_1^{(d)}, \dots, f_r^{(d)}$ form a basis of $W_d^{G_d}.$
        \item (\textbf{Consistency}) The basis elements project onto each other: $${\rm proj}_{W_d}\left(f_i^{(d+1)}\right) = f_i^{(d)} \quad \text{for all }i \in [r] \text{ and } d\in \N. $$
    \end{enumerate}
    \end{THM}
A sequence of polynomials satisfying these two properties is called a \emph{free basis}. The existence of a free basis enables us to encode an infinite sequence of ``consistent'' invariant polynomials $$\left\{p_{d}(x)=\sum_{i=1}^r \alpha_i f_i^{(d)}(x)\right\}_{d}$$ using only finitely many parameters $\alpha_{1}, \dots, \alpha_{r}$. Here, consistency means that polynomials with more variables $p_{d}$ and fewer variables $p_{d'}$ yield matching outputs when evaluated on low-dimensional inputs $x \in V_{d'}$ (we can evaluate $p_{d}(x)$ by setting to zero any components that do not appear in $V_{d'}$). We refer to such dimension-independent polynomial sequences as \emph{free polynomials}.
Common examples of free polynomials include the $\ell_{2}$-norm squared of a vector and the trace of a symmetric matrix, with many more appearing in optimization, signal processing, and particle systems  \cite{levin2023free}. The finite parameterizability of free polynomials makes them a powerful tool for computation and statistical inference across dimensions --- two aspects we explore next.

\paragraph{\bf Computing invariant bases and kernels} Although the bounds in Examples \ref{thm:Permutations}-\ref{theorem: point cloud rank} ensure that the rank of an invariant kernel is independent of its data dimension, it is not always the case that invariant kernels can be efficiently computed in high dimensions. Typical kernels used in applications, such as the Gaussian and standard polynomial kernels, are compactly expressible using inner products or vector norms and are therefore efficient to compute. Our next contribution is to show in two important cases how invariant kernels and free polynomial bases can be computed efficiently.

Free basis of invariant polynomials allows us to develop a procedure for evaluating invariant kernels. Let $\{V_d\}_{d\in\N}$ a sequence of vector spaces acted on by a sequence of groups $\{G_d\}_{d\in\N}.$ Suppose that $K_d:V_d\times V_d\to\R$ is a $G_d-$invariant kernel of degree $m$ and $\{f_1^{(d)}, \ldots, f_r^{(d)}\}_{d\in\N}$ form a free basis of invariants for $\{\R[V_d]_m^{G_d}\}_{d\in\N}.$ We begin by showing that the invariant kernel $K_d$ can be written as a symmetric bilinear form with respect to the free basis:
$$K(x, y) = F(x)^\T C F(y)$$
for all $x, y\in V_d,$ where $F(x) = \begin{bmatrix}
    f_1^{(d)}(x) & \ldots & f_r^{(d)}(x)
\end{bmatrix}^\T.$ Moreover, this matrix $C$ is unique and is independent of the dimension $d$ for all $d\geq m.$ Note that while the existence of such a matrix is guaranteed, determining its entries may be non-trivial.

For each of the sequences of vector spaces in Examples \ref{thm:Permutations}-\ref{theorem: point cloud rank}, a free basis of invariants can be formed by calculating the $\fS_d-$average of each monomial in $\R[V_d]_m.$ However, the number of operations required to evaluate each monomial's average is at least $|\fS_d|=d!$ and so forming this basis is not feasible for large values of $d.$ In general, it is not clear how to form free bases that can be computed with a numerical effort that is polynomial in the dimension $d$. With this in mind, we will show that there are indeed efficiently computable free bases for permutation and set-permutation invariant polynomials. Namely, we will show that \textit{elementary symmetric polynomials indexed by partitions} and \textit{polarized elementary symmetric polynomials indexed by vector partitions} form free bases in these two cases, respectively. The number of elements in these bases are given by the expressions in Examples \ref{thm:Permutations} and \ref{thm:Set-permutations}, respectively. %
Utilizing these bases, we obtain bounds on the numerical complexity of evaluating permutation and set-permutation invariant polynomial kernels that scales only  linearly in the dimension $d$. %

\begin{THM}[label={thm:complexity}]{(Computing $\fS_d$-invariant kernels for vectors and sets)}{}
\begin{enumerate}
    \item Let $K:\R^d\times\R^d\to\R$ be a permutation invariant polynomial kernel of degree $m.$ Assume that one has access to a matrix $C$ such that $K(x, y) = E_m(x)^\T C E_m(y)$ for all choices of $x, y\in\R^d,$ where $E_m(x)$ is the vector of elementary symmetric polynomials indexed by partitions of degree at most $m$. Then there exists an algorithm to evaluate $K(x,y)$ that uses
    $$O\left(md + m\exp\left(\sqrt{2m/3}\right)\right)$$
    floating point operations.
    \item Let $K:\R^{d\times k}\times\R^{d\times k}\to\R$ be a set-permutation invariant polynomial kernel of degree $m.$ Assume that one has access to a matrix $C$ such that $K(x, y) = P_m(x)^\T C P_m(y)$ for all choices of $x, y\in\R^{d\times k},$ where $P_m(x)$ is the vector of polarized elementary symmetric polynomials indexed by partitions of degree at most $m$. Then there exists an algorithm to evaluate $K(x,y)$ that uses
    $$O\left(m\exp\left(2\sqrt{2m/3}\right)\cdot k^{2mk}+\frac{dk^m}{m!}\right)$$
    floating point operations.
\end{enumerate}
\end{THM}

The proof of Theorem \ref{thm:complexity} amounts to determining the complexity of evaluating the relevant basis of invariant polynomials (see Propositions \ref{prop:permutation-complexity} and \ref{prop:set-complexity}), and so similar statements could be made for other groups given a set of efficiently computable free bases. However, although efficient computation of invariant kernels is intertwined with the efficient computation of invariant bases, the two problems are distinct. So, while invariant polynomial bases provide one method of computing invariant kernels in high dimensions, invariant kernels may be computed without explicitly evaluating the relevant basis. For example, in Illustration \ref{ex:set-tf}, we give an example of set-permutation invariant kernel that can be evaluated efficiently without needing to evaluate the basis of polarized elementary symmetric polynomials.

\paragraph{\bf Generalization of invariant polynomial regression across dimensions} As our final contribution, we leverage the existence of the basis prescribed by Theorem~\ref{thm:free-basis} to study polynomial regression problems where we aim to learn a sequence of fixed-degree invariant polynomials. Formally, we wish to recover a free polynomial\footnote{We drop the index $d$ as it can be inferred from the dimension of the input $x.$}
--- parameterized by $\alpha^{\star} \in \R^k$ via $p^{\star}(x) = \sum_{i=1}^k \alpha_i^{\star} f_i(x)$ with $f_1, \ldots, f_k$ a basis of $\R[V_d]_m^{G_d}$ --- from a set of noisy input-output pairs $(x_i,y_i)$ where $y_i = p^{\star}(x_i) + \varepsilon_i$ with the $\varepsilon$'s drawn i.i.d. with mean-zero and bounded variance. Importantly, the feature vectors $x$ might have different dimensions. We propose a least-squares-type estimator of the latent parameter via
\begin{equation}\label{eq:informal-estimator}
    \widehat \alpha = \mathop{\argmin}_{\alpha \in \R^k} \frac{1}{2n}\sum_{i = 1}^n \left( \sum_{j = 1}^k \alpha_i f_j(x_i) - y_i \right)^2,
\end{equation}
note that the basis elements $f_k$ are different depending on the dimension of their input. 
\begin{THM}[label={thm:minimax-informal}]{(Minimax optimality across dimensions - Informal)}{}
    Fix a sequence of polynomials $\{p^{\star}_d \in \R[V_d]^{G_d}\}$ parameterized by $\alpha^{\star} \in \R^k,$ a training set of input-output pairs $S = \{(x_1, y_1), \dots ,(x_1, y_n)\}$ and a test set $T = \{(x_1', y_1'), \dots (x_m', y_m')\}$ with the feature vectors in potentially different dimensions. Then, the estimator $\widehat \alpha$ in \eqref{eq:informal-estimator}, trained with $S$, yields a the sequence of invariant polynomials $$\left\{\widehat p_d = \sum_{i=1}^k \widehat \alpha_k f_k^{(d)} \in \R[V_d]^{G_d}\right\}$$
    that is minimax optimal with respect to the excess risk defined by any test set $T.$
    \end{THM}

\textbf{Outline.} The outline of the rest of the paper is as follows. The remainder of this section overviews the related literature. Section~\ref{section: preliminaries} provides the relevant background on invariant polynomial kernels and proves our main Theorem~\ref{theorem: rank is number of orbits}. Section~\ref{section: theorem} verifies each of the Examples~\ref{thm:Permutations}, \ref{thm:Set-permutations}, \ref{theorem: graphs rank}, and \ref{theorem: point cloud rank}. Section~\ref{sec:characters} overviews basic representation and character theory and proves Theorem~\ref{thm:calculate_dim_intro}. The next three sections concern connections to representation stability and generalization across dimensions: Section \ref{sec:stability_background} gives a background on representation stability, Section \ref{sec:generalization} discusses polynomial regression across dimensions and provides minimax optimal bounds for excess risk, and Section \ref{sec:bases} provides efficient methods for computing permutation and set-permutation invariant kernels. Finally, Section \ref{sec:numerics} details three numerical experiments that support the theory in this paper. Proofs from Section~\ref{section: preliminaries} and Section~\ref{sec:generalization} are given in Appendix~\ref{app:prelim-proofs} and Appendix~\ref{app:generalization} respectively.

\subsection{Related literature}

A number of recent works have examined the gain achieved by learning with invariant kernel methods. 

{\it Approximation and generalization gain in invariant methods.} Mei, Misiakiewicz, and Montanari \cite{mei2021learning} characterized the gain in both approximation and generalization error of certain invariant random feature and kernel methods. In particular, the assumed target functions are invariant under the action of some subgroup of the orthogonal group and that data is uniformly sampled from either a $d-$dimensional sphere of radius $\sqrt{d}$ or a $d-$dimensional hypercube. In both the underparametrized and overparametrized regimes, they showed that the gain from learning under invariances with kernel ridge regression is $d^\a$ for some parameter $\a$ that is dependent on the group of invariances. Several groups, such as both one and two-dimensional cyclic groups, satsify $\a=1$ yielding a linear gain in the dimension. %

{\it Sample complexity and generalization bounds under invariances.} Both \cite{tahmasebi2023exact} and \cite{bietti2021sample} discussed the gain in sample complexity and prove generalization bounds for kernel ridge regression when learning under invariances. The paper \cite{bietti2021sample} studied invariances under isometric actions of finite permutation groups on the sphere and \cite{tahmasebi2023exact} presented similar results for arbitrary Lie group actions on arbitrary compact manifolds, a much more general class of problems. In both works the statistical properties of the KRR estimator are controlled by the decay of the eigenvalues of certain operators. In \cite{tahmasebi2023exact}, the eigenvalues of the Laplace-Beltrami operator are controlled and in \cite{bietti2021sample} the eigenvalues of the covariance operator are controlled. A key tool used in both of these works is counting the dimensions of various classes of invariant functions: in \cite{bietti2021sample}, the proofs rely on comparing the dimensions of invariant and non-invariant spherical harmonics at varying degrees, and in \cite{tahmasebi2023exact} the arguments utilize a dimension counting theorem for functions on the quotient space $\cM/G$ where $\cM$ is a compact manifold and $G$ a Lie group. Though we are interested in a different question than in \cite{bietti2021sample, tahmasebi2023exact}, we use a similar strategy of determining the dimensions of invariant function spaces. Our work focuses on spaces of polynomials and in many of our concrete examples, elementary combinatorial arguments can be used to explicitly count and bound the dimensions of invariant polynomial spaces without being restricted to finite group actions on the sphere as in \cite{bietti2021sample} or having to calculate volumes and dimensions of quotient spaces as in \cite{tahmasebi2023exact}. In our examples, there is no dependence on the dimension of the data; instead, our results depend on the degree of the polynomials being considered.  

{\it Generalization gain of invariant predictors.} The work \cite{elesedy2021generalization} also discussed the benefit of invariances, though a different approach to enforcing invariance is taken. While our work, as well as the previously discussed works, consider kernels that are themselves assumed to be invariant under the action of some group, \cite{elesedy2021generalization} considered predictors that arise via normal KRR and are then averaged over a group to form an invariant function. For a given RKHS $\cH,$ the action of a group $G$ induces the decomposition $\cH = \bar{\cH}\oplus\cH_\perp$ where $\bar{\cH}$ is the subspace of $\cH$ consisting of all $G-$invariant functions. The so-called \textit{generalization gap} of the squared error loss between a predictor and its projection onto the invariant subspace $\bar{\cH}$ scales with the effective dimension of $\cH_\perp,$ which is calculated in terms of the eigenvalues of the integral operator mapping the space of square-integrable functions to $\cH.$ From this result, it is clear that having strong upper bounds on the dimension of invariant function spaces is important for understanding the benefit of learning with invariances in the worst case. Unlike \cite{mei2021learning, bietti2021sample, tahmasebi2023exact}, the generalization benefits in \cite{elesedy2021generalization} do not depend on properties of the underlying data distribution and are instead focused on the underlying structure of kernels, which is also the case for our work.

{\it Kernels with finite rank.} The previous works examine the role of invariance in decreasing the generalization error of kernel methods. This is different from our guiding question of how exploiting invariances can lead to a decrease in the rank of a kernel. The recent work \cite{cheng23finiterank} considers the generalization error of kernel ridge regression for kernels of finite rank. In particular, the authors proved sharp upper and lower bounds on the test error showed that variance is bounded linearly in the rank of the kernel. This suggests that low-rank kernels, such as those that arise from exploiting invariance, could be less prone to over-fitting data. 

{\it Permutation invariant kernels.} The paper \cite{Klus_2021}  defined \textit{antisymmetric} and \textit{symmetric} kernels on $\R^d,$ with their notion of symmetric kernels coinciding with our definition of $\fS_d-$invariant kernels. These antisymmetric and symmetric kernels were applied to problems in quantum chemistry and physics to exploit known symmetries in data. The authors showed that the degree of the feature space generated by an $\fS_d-$invariant polynomial kernel of degree $m$ is bounded by the sum $\sum_{j=0}^m p(j),$ which is comparable to our Example \ref{thm:Permutations}. However, in our verification of Example \ref{thm:Permutations}, we show that this is a special case of our main theorem, Theorem \ref{theorem: rank is number of orbits}. %
While there is some overlap between our results and the work in \cite{Klus_2021}, our arguments generalize to a much wider class of groups. Moreover, the authors of \cite{Klus_2021} relied on either averaging kernels over the whole group $\fS_d$ and/or evaluating large number of hyperpermanents to compute the value of $\fS_d-$invariant kernels. This becomes infeasible for high dimensional problems since the size of the permutation group is exponential in the dimension. In contrast, we utilize the theory of symmetric polynomials to provide efficient methods of computing $\fS_d$-invariant kernels for large values of $d.$

\paragraph{\it Representation stability and generalization across dimensions.} Originally introduced in algebraic topology \cite{church2013representation, church2014representation}, representation stability has since been applied to study limits of several mathematical objects \cite{sam2017grobner, sam2015stability, sam2016gl,conca2014noetherianity,alexandr2023moment}. Closer to our setting, \cite{levin2023free} used representation stability to parameterize sequences of convex sets of increasing dimension, and \cite{levin2024any} employed it to design equivariant neural networks that accommodate inputs of arbitrary dimension. To our knowledge, no generalization bounds across dimensions have been established. However, in the context of Graph Neural Networks (GNNs) \cite{gori2005new, micheli2009neural, bruna2013spectral, wu2020comprehensive}, there has been noteworthy progress on the related notion of transferability. GNNs can handle graphs of arbitrary sizes, and as these graphs increase in size and converge to a limiting object known as a {\em graphon}, a natural question arises: do the GNN outputs also converge? Recent results \cite{ruiz2021graph, maskey2023transferability} show that they do, ensuring that the outputs on “related” graphs of different sizes remain close to one another.

\section{Polynomial kernels and their rank}\label{section: preliminaries}

We begin by recording some background on kernels, following the standard texts on the subject \cite{scholkopf2002learning,shawe2004kernel}. To this end, let $V$ be a finite-dimensional vector space with a fixed basis. 
A {\em kernel} $K\colon V\times V\to\R$ is any symmetric function of its two arguments. The main use of kernels in applications is to measure pairwise similarity between data points 
$x_1,\ldots, x_n\in V$ via the $n\times  n$ symmetric matrix  $\mathbf{K}=[K(x_i,x_j)]_{i,j=1}^n$. The kernel $K$ is called {\em positive semidefinite (PSD)}  if the matrix $\mathbf{K}$ is positive semidefinite for any data points. Typical examples of PSD kernels on $\R^d$ are the polynomial kernel $K(x,y)=(1+x^{\top} y)^m$, the Taylor features kernel $K(x,y)=\exp(-(\|x\|^2+\|y\|^2)/2\sigma^2)\cdot\sum_{j=1}^m\tfrac{\langle x,y\rangle^j}{\sigma^{2j}j!}$, and the Gaussian kernel $K(x,y)=\exp(-\|x-y\|^2/2\sigma^2),$  where $\s$ is some fixed bandwidth parameter.

A kernel $K$ is said to admit a {\em rank $r$-factorization} if there exist maps $\phi,~ \psi\colon V\to\R^r$---called {\em feature maps}---satisfying 
$$K(x, y)=\langle\phi(x), \psi(y)\rangle\qquad \forall x,y\in V.$$ 
The \textit{rank of $K$}, denoted $\rank K$, is the minimal $r$ such that $K$ admits a rank $r$-factorization. Importantly, the rank of $K$ upper-bounds the rank of the corresponding kernel matrix:
$$\rank \mathbf{K}\leq \rank K.$$

\subsection{Group invariance and symmetrization.}\label{sec:gi_sym}
We will be interested in the interplay between kernel rank and group invariance. Namely, let $G$ be a group acting on $V$.  A kernel $K$ on $V$ is called {\em $G$-invariant} if the following equality holds:
\begin{equation}\label{eqn:g_invar_pre}
  K(gx,hy)=K(x,y)\qquad \forall x,y\in V, ~g,h\in G.  
\end{equation}
Notice that in $\eqref{eqn:g_invar_pre}$, the group $G$ acts independently on the two coordinates $x$ and $y$. This should be contrasted with the much weaker property of $G$-{\em translation invariance}, which stipulates that \eqref{eqn:g_invar_pre} holds only with $h=g$. Importantly, $G$-invariant kernels are well-defined on the orbit space $V/G$, i.e., the set of orbits of $V$ under the action of $G$.

Any kernel can be symmetrized with respect to any finite group $G$ simply by averaging:
$$K^*(x, y):=\tfrac{1}{|G|^2}\sum_{g,h\in G} K(gx,hy).$$
The analogous construction for infinite groups requires averaging with respect to a translation invariant measure on $G$, called the Haar measure. The precise formalism is as follows. 

\begin{assumption}[Haar measure]\label{ass:group_assump}
Suppose that $G$ is a  compact Hausdorff topological group, equipped with the Borel $\sigma$-algebra and the right Haar measure $\mu$ normalized to satisfy $\mu(G)=1$. Suppose moreover that $G$ acts continuously on $V$.
\end{assumption} 

To shorten the notation, we will use the symbol $\E_{g}$  to mean expectation with respect to the Haar measure $\mu$ on $G$. A key property of the Haar measure is the invariance
$\E_g \phi(gh)=\E_g \phi(h)$ 
for any measurable function $\phi$ and group element $h\in G$. In particular, any continuous function $f$ yields the $G$-invariant symmetrization   
$$f^*(x) := \mathop \E_g f(gx).$$
Similarly, any continuous kernel $K$ on $V$ yields the new $G$-invariant kernel by averaging
$$K^*(x, y) := \mathop\E_{g,h} K(gx, hy).$$

The following lemma is standard; we provide a short proof in Section~\ref{sec:proof_lem_inv}.
\begin{lemma}\label{lem:inv_kern}
The kernel $K^*$ is $G$-invariant. Moreover, if $K$ is PSD then so is  $K^*$.
\end{lemma}

The main thrust of our paper is to show that the rank of the symmetrized kernel $K^*$ is often much smaller than the rank of $K$.

\subsection{Invariant polynomials and kernel rank} 
In the rest of this paper, we focus on kernels that are polynomials of a fixed degree. To be precise, let $\R[V]_m$ be the vector space of all polynomials over $V$ with total degree at most $m.$ Any linear action of a group $G$ on $V$ promotes to a linear action of $G$ on the polynomial ring $\R[V]_m$ by 
$$(g\cdot f)(x)=f(g^{-1}x).$$
The symbol $\R[V]^G_m$ will denote the set of $G$-invariant polynomials $f \in \R[V]_m$, meaning
$$f(gx)=f(x)\qquad \forall x\in V,~g\in G.$$
The dimension of $\R[V]^G_m$,  a linear subspace of $\R[V]_m$,  will play a key role in the paper.

\begin{definition}[Polynomial kernel]
    A kernel $K\colon V\times V\to\R$ is a \textit{polynomial kernel} if $K(x,y)$ is a polynomial for all $x,y \in V$.
  The \textit{degree} of $K(x,y)$ is its maximal degree 
 in $x$ (or equivalently in $y$, by symmetry).
\end{definition}

\begin{assumption}\label{ass:setting} The following hold.
    \begin{enumerate}
\item  The function $K$ is a degree $m$ polynomial kernel on $V$.
\item The group $G$ acts on $V$ satisfying Assumption~\ref{ass:group_assump}.
\item The set of polynomialrs $\{ f_1,\ldots,f_N \}$ forms a basis of $\R[V]_{m}$.
\end{enumerate}
\end{assumption}

It will be convenient to express $K$ as a bilinear form with respect to the basis $\{f_i\}_{i=1}^N$---the content of the following standard lemma. A short proof can be found in Section~\ref{sec:proof_lembilin}.

\begin{lemma}[Bilinear form]\label{lemma: C is unique}
    Suppose that Assumption~\ref{ass:setting} holds. Then there exists a unique matrix $C\in \R^{N\times N}$, which is necessarily symmetric, such that for all $x,y\in V$, it holds: 
    \begin{align}
    K(x, y)&= \F(x)^\T \cdot C\cdot \F(y),\label{eqn:K_as_lin_form}\\
    K^*(x, y) &= \F^*(x)^\T \cdot C\cdot \F^*(y)\label{eqn:K_as_lin_form_sym},
\end{align}
where we set $\F(x)=[f_1(x),\ldots,f_N(x)]^\T$ and $\F^*(x)=[f_1^*(x),\ldots,f^*_N(x)]$.
\end{lemma}

Conveniently, the rank of the matrix $C$ in equation \eqref{eqn:K_as_lin_form} coincides with the rank of $K$; the following lemma is proved in \cite[Proposition 3.6]{altschuler-parrilo} and \cite[Corollary 3.8]{altschuler-parrilo} in greater generality.

\begin{lemma}[Exact rank formula]\label{lemma: rank k = rank C}
   Suppose Assumption~\ref{ass:setting} holds. Fix a matrix $C$ satisfying \eqref{eqn:K_as_lin_form}. Then $\rank K = \rank C.$  Moreover, there exist polynomial maps $\varphi,\,\psi\colon V\to\R^{\rank K}$ satisfying $K(x,y)=\langle \varphi(x),\psi(y)\rangle$ for all $x,y\in V$. In addition, $K$ is PSD if and only if $C$ is PSD, and then we can set $\psi=\varphi$. 
\end{lemma}

The following is the main theorem of the section, and it verifies the estimate \eqref{eqn:rank_est_basis0}. It shows that we may upper bound the rank of a polynomial kernel by the dimension of the vector space  $\R[V]_m^G$, consisting of $G$-invariant polynomials of degree at most $m$.

\begin{theorem}[Rank and polynomial invariants] \label{theorem: rank is number of orbits}
Suppose Assumption \ref{ass:setting} is satisfied. Then, the following inequalities hold: \begin{equation}\label{eqn:rank_est_basis}
\rank K\leq \dim\left({\rm span}\{f_1^*, \ldots, f_N^*\}\right)=\dim(\R[V]_m^G).
\end{equation}
\end{theorem}

Note that $\dim\left({\rm span}\{f_1^*, \ldots, f_N^*\}\right)$ is upper bounded by $|\{f_1^*, \ldots, f_N^*\}|,$ the number of  distinct symmetrizations of the basis elements $f_1, \ldots, f_N$. %
While it is not true in general %
that the number of $G-$symmetrizations of basis elements coincides with $\dim(\R[V]_m^G)$ (see Illustration~\ref{counterexample}), the two quantities do agree for permutation groups applied to the monomial basis, and so for our examples, it suffices to count the orbits of monomials under the relevant permutation actions. This is the content of the following lemma and illustration.

\begin{example}\label{counterexample}
    Suppose $\Z/4\Z$ acts on $\R^2$ by $90^\circ$ counterclockwise rotation, where the primitive generator sends $(a,b) \mapsto (-b,a)$. Consider  the space of degree 2 polynomials $\R[x_1, x_2]_2$ with the basis $\{1, x_1, x_2, x_1^2, x_1^2+x_2^2, x_1x_2\}.$ For all $g\in\Z/4\Z,$ the equality $g\cdot(x_1^2+x_2^2) = x_1^2+x_2^2$ holds and so $(x_1^2+x_2^2)^* = x_1^2+x_2^2.$ Note that $(x_1^2)^* = \frac{1}{4}(x_1^2+x_2^2+x_1^2+x_2^2) = \frac{1}{2}(x_1^2+x_2^2).$ It follows that $(x_1^2)^*$ and $(x_1^2+x_2^2)^*$ are not equal but are linearly dependent and so the number of $G-$symmetrizations of the basis is not equal to the dimension of the span of the symmetrizations.
\end{example}

\begin{lemma}\label{lemma: rank = number of orbits}
    Suppose $G\subset \fS_d$ acts on $V=\R^d$ by permutating coordinates and let $f_1, \ldots, f_N$ be the monomial basis of $\R[V]_m.$ Then,
    \begin{align*}
        |\{f_1^*, \ldots, f_N^*\}| = \dim(\Span\{f_1^*, \ldots, f_N^*\}) = \dim(\R[V]_m^G).
    \end{align*}
\end{lemma}

\begin{proof}
    Let $f_{i_1}^*, \ldots, f_{i_r}^*$ be representatives of all $G-$orbits. As $G$ acts via permutation and each $f_j$ is a monomial, each $f_{i_j}^*$ is a linear combination of monomials, with each monomial in $\R[V]_m$ appearing in exactly one $f_{i_j}^*$. Since all monomials are linearly independent, it follows that any nontrivial linear combination $\sum_{j=1}^r\a_if_{i_j}^*$ is not equal to zero. Thus, the set $\{f_{i_1}^*, \ldots, f_{i_r}^*\}$ is linearly independent and the desired equalities hold.
\end{proof}

\section{The four examples}\label{section: theorem}
Concrete applications of Theorem~\ref{theorem: rank is number of orbits} require a few combinatorial constructions, which we now recall and use to prove the bounds in Examples~\ref{thm:Permutations}-\ref{theorem: point cloud rank}.

\subsection{Partitions and compositions.}\label{sec:notation_part_comp}

Fix some positive integer $j\in \mathbb{N}.$ A \textit{composition} of $j$ is an ordered sequence of strictly positive integers $(\a_1, \a_2, \ldots)$  that sum to $j.$ If the values $\a_i$ are required to only be nonnegative, then the sequence is called a {\em weak composition}. If the number of terms in the sequence is $d$, then the sequence is called a \textit{(weak) composition of $j$ into $d$ parts}. We let the set $C(j)$ consist of all compositions of $j$ and we let $C_d(j)$ consist of all weak compositions of $j$ into $d$ parts. 
The cardinality of these two sets will be written as 
 $$c(j):={\rm Card}~ C(j)\qquad \textrm{and}\qquad c_d(j):={\rm Card} ~C_d(j).$$ 
In particular, $c(j)$ and $c_d(j)$ can be explicitly written as 
$$c(j)=2^{j-1}\qquad \textrm{and}\qquad c_d(j)={j+d-1\choose j}.$$
A \textit{partition of $j$} is a composition of $j$ with ordered elements $\a_1\geq \a_2\geq\ldots.$ A \textit{partition of $j$ into $d$ parts} is defined analogously. We denote by $P(j)$ the set of all partitions of $j$ and by $P_d(j)$ the set of all partitions of $j$ with at most $d$ parts. The cardinality of these two sets will be written as 
 $$p(j):={\rm Card}~ P(j)\qquad \textrm{and}\qquad p_d(j):={\rm Card}~P_d(j).$$
We will use the shorthand $\a\vdash j$ to mean that $\a$ is a partition of $j$, while the number of parts of $\a$ will be denoted by $|\a|$. It is sometimes convenient to identify the partition $\a\vdash j$ as a multiset $\a=\{1^{\mu_1}, 2^{\mu_2}, \ldots, j^{\mu_j}\}.$ We denote the multipilicities of parts of $\a$ by the vector $\mu(\a) = (\mu_1(\alpha), \ldots, \mu_j(\alpha))$ and the product of factorials of the multiplicities by $\eta(\a) = \prod_{i=1}^j\mu_i(\a)!.$

By convention, the number of compositions and partitions of zero is equal to $1$.

\begin{example}
    As a concrete example, $4$ can be partitioned in $5$ distinct ways:
$$
4\qquad 3+1\qquad
2+2\qquad
2+1+1\qquad
1+1+1+1.
$$
Consequently, we have $p_1(4)=1$, $p_2(4)=3$, $p_3(4)=4$, and $p_4(4)=5$. Arbitrarily permuting the elements within each partition yields a composition. Padding each composition with zeros yields a weak composition.
\end{example}

We will now give a simple direct proof of Theorem~\ref{thm:Permutations}. The result also follows from facts about the invariant ring $\R[V]^{{\fS}_d}$ that will require introducing algebraic results that are not necessary in this paper, see for example \cite{fulton-harris}.

\subsection{Verification of Example~\ref{thm:Permutations}}
We plan to apply Theorem~\ref{theorem: rank is number of orbits} and Lemma~\ref{lemma: rank = number of orbits}. Assume that $V = \R^d$ and  fix a monomial basis of $\R[V]_m$. Each monomial in $\R[V]_m$ can be written as $x^{\alpha}:=x_1^{\a_1}x_2^{\a_2}\ldots x_d^{\a_d}$ for some set of nonnegative integers $\a_1, \ldots, \a_d.$ The degree of the monomial is $j:=\sum_{i=1}^d\a_i\leq m$. Since the action of $\fS_d$ on $\R^d$  is given by coordinate permutations, the action of $\fS_d$ of $\R[V]_m$ is given by a permutation of the monomials. Namely, for all $\s\in\fS_d,$ the equality holds:
\begin{align*}
    \s\cdot x_1^{\a_1}x_2^{\a_2}\ldots x_d^{\a_d} &= x_{\s(1)}^{\a_1}x_{\s(2)}^{\a_2}\ldots x_{\s(d)}^{\a_d}=x_1^{\a_{\s(1)}}x_2^{\a_{\s(2)}}\ldots x_d^{\a_{\s(d)}}.
\end{align*}
Consequently, the $\fS_d-$orbit of a degree $j$ monomial $x^\alpha$ can be labeled by a partition of $j$ into at most $d$ parts. Thus, the number of $\fS_d-$orbits of the monomial basis of $\R[V]_m$ is $\sum_{j=0}^m p_d(j)$. Noting $p_d(j)\leq p(j)$ when $d\leq j$ and $p_d(j) = p(j)$ for $d>j$ and applying Theorem~\ref{theorem: rank is number of orbits} completes the proof.

\subsection{Verification of Example~\ref{thm:Set-permutations}}
Recall that we are interested in  the group of permutations $\fS_d$ acting on matrices $X\in\R^{d\times k} = V$ by permuting rows. Fix a monomial basis of $\R[V]_{m}$. It is straightforward to see that each monomial $X^{\beta}$ can be labeled by a matrix of nonnegative integers $\beta\in \R^{d\times k}$ whose $i$'th row $\beta_i$ is a weak composition of some integer $j_i$ into $k$ parts and we have $\sum_{i=1}^d j_i=j\leq m$. By permuting the rows, we may find at least one monomial in the $\fS_d$-orbit of $X^{\beta}$ so that $j_1\geq j_2\geq\ldots\geq j_d.$ Note that when $d\geq m,$ all nonzero entries of $\b$ will be in the first $m$ rows. We identify the row sums $(j_1, \ldots, j_d)$ with a partition of $j$  into at most $d$ parts. Identifying the partition $\a\vdash j$ with the multiset $\a=\{1^{\mu_1}, \ldots, j^{\mu_j}\},$ each $\fS_d-$orbit representative with row sums $\alpha$ corresponds to a distinct choice of $\mu_i$ weak compositions of $i$ for each $i\leq j.$ It follows that the number of $\fS_d-$orbits is given by $\sum_{j=0}^m\sum_{\a\vdash j}\prod_{i=1}^j{c_k(i) + \mu_i(\a)-1\choose \mu_i(\a)}$ and is independent of $d$ for all $d\geq m,$ as claimed. An application of Theorem~\ref{theorem: rank is number of orbits} completes the proof.

\subsection{Verification of Example~\ref{theorem: graphs rank}}

    Each monomial $X^{\beta}\in \R[V]_m$ corresponds to a matrix $\beta\in \R^{d\times d}$  with each entry $\beta_{ij}$ being the exponent of $x_{ij}$ in the monomial. 
    Since every adjacency matrix is symmetric, in each monomial, $x_{ij}$ can be replaced by $x_{ji}$ for all pairs $i,j.$ It follows that we need only consider symmetric matrices of exponents since we can replace the $ij^{th}$ entry of the exponent matrix $\beta$ with $\frac{1}{2}(\beta_{ij}+\beta_{ji})$, making some entries of the matrix odd multiples of $\frac{1}{2}$ and not integers.

    Next, let $j$ be the degree of $X^\beta$ and let $\alpha$ be the vector of row sums of $\beta$, that is $\alpha_i=\sum_{j}\beta_{ij}$. By simultaneously permuting the rows and columns of $X$, equivalently those of $\beta$, we can find a monomial in the $\fS_d-$orbit of $X^\beta$ such that $\alpha$ is ordered, meaning $\alpha_1\geq \alpha_2\ldots\geq \alpha_d$. 
    When $d\geq 2m,$ we must have that $\a_i = 0$ for all $i>2m$ and so we need only consider $2m\times 2m$ matrices. Noting that $2\alpha$ is a partition of $2j$, we deduce that the number of  $\fS_d-$orbits is upper bounded by $\sum_{j=0}^m \sum_{\a\vdash 2j}q_\a$ where $q_\a$ is the number of symmetric non-negative integer matrices with row sums equal to $\a$ %
    and is independent of $d$ when $d\geq 2m.$ An application of Theorem~\ref{theorem: rank is number of orbits} completes the proof.
    
\subsection{Verification of Example~\ref{theorem: point cloud rank}}
    Recall that the data points are matrices $X\in \R^{d\times k}$. The first fundamental theorem of invariant theory for the orthogonal group \cite[Theorem 2.9.A]{weyl1946classical} shows that any degree $m$ polynomial in $\R[V]$ that is invariant under simultaneous orthogonal action on the rows $x_i$ is a degree $m/2$ polynomial of the pairwise inner-products $\langle x_i,x_j\rangle$. That is, any such polynomial  $f(x)=\Tilde{f}(A(x))$ where $A(x)=[\langle x_i,x_j\rangle]\in \R^{d\times d}$. If the polynomial $f$ is in addition invariant under relabeling of the rows, we can find a representing element for its $G$-orbit such that $\Tilde{f}$ is invariant under conjugation by permutations. A completely analogous argument
    to the proof in Theorem~\ref{theorem: graphs rank} shows that the number of such invariants is bounded by $\sum_{j=0}^{m/2} \sum_{\a\vdash 2j}q_{\a}$. Invoking Theorem~\ref{theorem: rank is number of orbits} completes the proof.

\section{Proof of Theorem~\ref{thm:calculate_dim_intro} and Consequences}\label{sec:characters}

 In this section, we will prove Theorem~\ref{thm:calculate_dim_intro}. The arguments will be based on the theory of {\em characters} from representation theory. This is a systematic approach that applies to any finite group or, more generally, to
 any compact Lie group $G$ acting continuously on an Euclidean space.
See \cite{folland-abstract-harmonic-analysis, SerreBook, barry-simon},  for the representation theory facts used in this section. 

Suppose that a group $G$ acts linearly on a vector space $V$. This action can be summarized by a map (group homomorphism) $\rho \colon G \rightarrow \GL(V)$, called a {\em representation of $G$}. By choosing a basis for $V$, we may identify $\rho(g)$ with a matrix $M_g$ so that for any $x \in V$
$$gx := \rho(g)(x) = M_g x.$$

\begin{definition}[Character]
The {\em character} of $\rho$ is the function $\chi:G \rightarrow \R$ defined as $$\chi(g) := \Tr(M_g).$$
\end{definition}

The main use of characters for us is to compute the dimension of the invariant space
$$V^G:=\{x\in V: gx=x ~~\forall g\in G\}.$$
This is the content of the following lemma, which is an easy consequence of \cite[Corollary 2.4.1]{SerreBook}. 

\begin{lemma}[Invariants and characters] \label{lem:character formula}
Let $G$ be a compact Lie group acting linearly and continuously on a vector space $V$. Then,
\begin{equation}\label{eqn:invariant_dim}
    \dim (V^G) = \mathop\E_{g\sim \operatorname{Unif}G} \chi(g),
    \end{equation}
    where the expectation is taken with respect to the Haar probability measure.
\end{lemma}

In particular, one can in principle estimate $\dim (V^G)$ by sampling elements $g\in G$ according to the Haar measure. In the case when $G$ is a finite group acting linearly on a vector space $V$, the expression \eqref{eqn:invariant_dim} simply becomes 
$$\dim (V^G) = \tfrac{1}{|G|}\sum_{g\in G} \chi(g)=\Tr\left(\tfrac{1}{|G|}\sum_{g\in G} M_g\right).$$

We next aim to apply Lemma~\ref{lem:character formula} in order to compute $\dim (\R[V]^G_m)$.  In order to do so, we need to describe how the representation $\rho$ of $G$ on $V$ promotes to a representation $\rho'$ of $G$ on $\R[V]_m$. 
Let $N$ denote the size of the (monomial) vector space basis of $\R[V]_m$ .
Since $G$ acts linearly on $\R[V]_{m}$  we may write
$$ g\cdot f=M'_g\cdot [f],$$
where $[f]\in \R^{N}$ is the vector of coefficients of $f$ in the monomial basis and $M'_g$ is an $N\times N$ matrix. Applying Lemma~\ref{lem:character formula} therefore yields the expression 
\begin{equation}\label{eqn:poly_lifting}
\dim (\R[V]_m^G) = \mathop\E_{g\sim \operatorname{Unif} G} \Tr(M'_g).
\end{equation}
Let us look at a concrete example that illustrates the expression \eqref{eqn:poly_lifting}.

\begin{example} \label{ex:n=1,d=2,S2}
    Consider the permutation group $G = \mathfrak{S}_2 = \{e,g\}$ acting on $\R^2$, where $e$ is the identity permutation and $g=(12)$. Then the representing matrices are
    $$\rho(e) = \begin{bmatrix} 1 & 0 \\ 0 & 1 \end{bmatrix}, 
    \quad \text{and} \quad \rho(g) = \begin{bmatrix} 0 & 1 \\ 1 & 0 \end{bmatrix}.
    $$
    Picking $m=2$, the space $\R[x_1,x_2]_2 = \textup{span} \{1,x_1,x_2, x_1^2, x_1x_2, x_2^2\}$ has dimension $N = 6$. The  representation $\rho'$ of $G$ on $\R[x_1,x_2]_2$ has the form
    $$M'_e = \begin{bmatrix}
        1 & 0 & 0 & 0 & 0 & 0 \\
        0 & 1 & 0 & 0 & 0 & 0 \\
        0 & 0 & 1 & 0 & 0 & 0 \\
        0 & 0 & 0 & 1 & 0 & 0 \\
        0 & 0 & 0 & 0 & 1 & 0 \\
        0 & 0 & 0 & 0 & 0 & 1,
    \end{bmatrix}, \quad \text{and}\quad M'_g = \begin{bmatrix} 
     1 & 0 & 0 & 0 & 0 & 0 \\ 
     0 & 0 & 1 & 0 & 0 & 0 \\ 
     0 & 1 & 0 & 0 & 0 & 0 \\
     0 & 0 & 0 & 0 & 0 & 1 \\ 
     0 & 0 & 0 & 0 & 1 & 0 \\
     0 & 0 & 0 & 1 & 0 & 0
     \end{bmatrix}.$$
    The character $\chi'$ of $\rho'$ is then given by $\chi'(e)=\Tr(\rho'(e)) = 6$ and $\chi'(g)=\Tr(\rho'(g)) = 2$. By \eqref{eqn:poly_lifting}, we expect $\dim (\R[x_1,x_2]^G_2)=4$. Indeed, a basis for $\R[x_1,x_2]^G_2$ is given by $\{1, x_1+x_2, x_1^2+x_2^2, x_1x_2\}$.
\end{example}

More generally, in order to apply \eqref{eqn:poly_lifting} we 
need to compute the trace of the representing matrix $M'_g$.
This can be explicitly calculated directly from the action of $G$ on $V$. Namely, as shown in \cite[Equation 8]{induced-matrices-1} and \cite[Section 2.12]{induced-matrices-2}, %
the entries of $M'_g$ are indexed by 
pairs of weak compositions $\alpha,\beta\in C_d(k)$ of all integers $k\leq m$ into $d$ parts. In particular, the diagonal entries of $M_g'$ are given by 
\begin{equation}\label{eqn:lift_diag}
    (M'_g)_{\alpha\alpha} = \frac{\per M_g[\a]}{{\eta(\a)}}.
\end{equation}
Here, $M_g(\a)$ is the principle submatrix of $M_g\in \R^{d\times d}$ indexed by $\a\in \R^d$, the symbol $\per M_g[\a]$ denotes the permanent of the matrix  $M_g[\a]$, and $\eta(\a)$ is the product of the factorials of the multiplicities of entries in $\a$. Summing the expressions \eqref{eqn:lift_diag} across $\alpha$ yields the second equality in \eqref{eqn:sophisticated}.

Alternatively, we may express the trace of $M_g'$ as the sum of its eigenvalues. To this end, let $\l_1, \ldots, \l_d$ be the eigenvalues of $M_g.$ Then it is shown in \cite[Section 2.15.14]{induced-matrices-2} and \cite[Section 3]{induced-matrices-1} that the eigenvalues of $M'_g$ are $\l'_\a = \prod_{i=1}^d\l_{i}^{\a_i}$ where $\a\in C_d(k)$ is some weak composition of $k\leq m$ into $d$ parts. Summing the eigenvalues yields the first equality in \eqref{eqn:sophisticated}.

A direct consequence of Lemma \ref{lem:character formula} is a lower bound on the dimension $\dim(\R[V]_m^G)$ of invariant polynomial spaces where $V$ is $d-$dimensional and $G\subset\fS_d$ acts by permuting basis elements of $V$.

\begin{corollary}
    Suppose that the finite group $G\subset \fS_d$ (with identity element 1) acts on the $d-$dimensional vector space $V$ by permuting basis elements. Then, the inequality holds:
    \begin{align*}
        \dim\R[V]_m^G >\frac{\dim\R[V]_m}{|G|} > \frac{d^m}{m^m|G|}.
    \end{align*}
\end{corollary}

\begin{proof}
    Note that since 1 is the identity element of $G,$ the action of $1$ on $\R[V]_m$ is represented by the identity matrix and thus $\chi'(1) = \dim\R[V]_m$. As $V$ is $d-$dimensional, the dimension is equal to the number of monomials in $d$ variables with degree at most $m:$ $$\chi'(1) = \dim\R[V]_m = \sum_{j=0}^m{m+d-1\choose d-1}.$$
    Moreover, since $G$ acts by permuting basis elements of $V,$ $G$ acts on $\R[V]_m$ by permuting monomials and thus each element $g\in G$ is represented by some permutation matrix. In particular, $\chi'(g)$ is equal to the number of monomials fixed by $g.$ As all constant monomials are fixed by each element $g\in G,$ the inequality $\chi'(g)\geq 1$ holds for each $g.$ We successively compute
    \begin{align*}
        \dim\R[V]_m^G &= \frac{1}{|G|}\sum_{g\in G}\chi'(g) \\
        &>\frac{1}{|G|}\chi'(1)\\
        &=\frac{1}{|G|}\dim\R[V]_m\\
        &> \frac{1}{|G|}{m+d-1\choose d-1}.
    \end{align*}
    Noting that ${m+d-1\choose d-1}> \frac{d^m}{m^m}$ completes the proof.
    \end{proof}

It follows that in order for the dimension $\R[V]_m^G$ to stabilize, we must have that the size of the group $|G|$ scales at approximately the same rate as $\dim\R[V]_m.$ As another illustration of Lemma \ref{lem:character formula}, we now use it to compute $\dim\R[V]_m^G$ for a cyclic group $G\subset \fS_d$. %

\begin{lemma}[Invariants under cyclic permutations]\label{lemma: cyclic invariant dimension}
    Let $G = \Z/d\Z$ act on $V=\R^d$ by cyclic permutation. Then, assuming $m<d$, the following equality holds: $$\dim(\R[V]_m^G) = 1+\frac{1}{d}\left(\sum_{r:~ s={\rm gcd}(r,d)>1}\sum_{j=1}^{\lfloor ms/d\rfloor}{j+s-1 \choose s-1}\right).$$
\end{lemma}

\begin{proof}
    Let $\s$ be a $d-$cycle generating $\Z/d\Z.$ Then, the equality
    \begin{align*}
        \dim\R[V]_m^G &= \frac{1}{d}\sum_{r=1}^{d}\chi'(\s^r) 
    \end{align*}
    holds by Lemma \ref{lem:character formula}. So, we need to calculate the trace of each element $\s^r$ acting on $\R[V]_m.$ Note that each element $\s^r$ permutes the monomials in $\R[V]_m$ and so the trace $\chi'(\s^r)$ is equal to the number of monomials fixed by the permutation $\s^r.$ As the monomial 1 is fixed by all permutations, we will now focus on monomials of degree at least 1.

    Note that for the $d-$cycle $\s,$ a monomial is fixed if and only if it is of the form $x_1^\a x_2^\a \ldots x_d^\a$ for some integer $\a.$ This monomial has total degree $d\cdot\a.$ By assumption the dimension $d$ is larger than the degree $m$ of polynomials being considered and so no nontrivial monomials are fixed by $d-$cycles. Thus, $\chi'(\s) = 1.$
    Moreover, for every $r$ with ${\rm gcd}(r,d)=1$ we have that $\s^r$ is also a $d-$cycle and so $\chi'(\s^r) = 1$.

    If ${\rm gcd}(r,d)>1$ the permutation $\s^r = \t_1\ldots\t_s$ decomposes as a product of $s:={\rm gcd}(r,d)$ disjoint cycles of length $\ell:=d/s.$ In order for a monomial to be fixed by $\s^r,$ exponents corresponding to the entries of each cycle $\t_i$ must all be equal and so each monomial of degree $j$ fixed by $\s^r$ corresponds to a weak composition of $j$ into $s$ parts such that every part is divisible by $\ell.$ Let $C_{s}^{\ell}(j)$ be the set of weak compositions of $j$ into $s$ parts divisible by $\ell.$ Note that in order for $C_s^{\ell}(j)$ to be non-empty, we must have that $j$ is divisible by $\ell.$ Using a simple stars and bars argument, we have $|C_s^{\ell}(j)|=|C_s(j/\ell)|={j/\ell + s-1 \choose s-1}.$ It follows that the equality
    \begin{align*}
        \chi'(\s^r) &= 1 + \sum_{j=1}^{m} |C_s^{\ell}(j)| = 1 + \sum_{j=1}^{\lfloor m/\ell \rfloor}{j + s-1 \choose s-1},
    \end{align*}
    holds. Thus, altogether we have that the equality
    \begin{align*}
        \dim \R[V]_m^G &= \frac{1}{d}\sum_{j=1}^d \chi'(\s^j)
        = \frac{1}{d}\left(d + \sum_{r:~ s={\rm gcd}(r,d)>1}\sum_{j=1}^{\lfloor ms/d\rfloor}{j + s-1\choose s-1}\right)
    \end{align*}
    holds as desired.
\end{proof}

We can now use Lemma \ref{lemma: cyclic invariant dimension} to calculate $\dim\R[V]_m^{\Z/d\Z\times\Z/d\Z}$ which bounds the rank of kernels on the space $V=\R^{d\times d\times k}$ of $d\times d$ images. 
Images are represented as 2D arrays $X=(x_{ij})_{ij=1}^d$ of pixels where each entry $x_{ij}$ is a vector in $\R^k$ for some dimension $k.$ For example, grayscale images can be represented as $d\times d$ matrices where the value of each entry determines how dark the corresponding pixel is ($k=1$), and RGB images are represented as 2D arrays with entries in $\R^3$ determining the red, green, and blue values for each pixel ($k=3$).

Setting $k=1$, note that we can capture the effect of translating a $d \times d$ image in an infinite plane by the action of $\Z/d\Z\times\Z/d\Z$ on the rows and columns of the image so that the translated image stays in the original $d \times d$ frame. This is important since 
a kernel may only take in a $d\times d$ dimensional input.
 Functions of images should be invariant under the set of translations since any translation of an image does not change its content. 
 The group of translations is $\Z/d\Z\times\Z/d\Z$ with the 2D cyclic permutation $(\s, \t)\in\Z/d\Z\times\Z/d\Z$ acting on $X=(x_{ij})_{i,j=1}^d$ by $(\s, \t)\cdot X=\s X\t^\T.$ Note that while we concentrate on grayscale images, the same argument could be used to give an analogous bound for the rank of kernels on other spaces of images, such as RGB images.

\begin{theorem}[2D-cyclic permutations]\label{theorem: image rank}
Let the group $G=\Z/d\Z\times\Z/d\Z$ act on matrices $V=\R^{d\times d}$ by conjugation, that is $(\sigma,\tau)\cdot X=\sigma X\tau ^\top.$ Then, the following equality holds:
    \begin{align*}
        \dim\R[V]_m^G &= \frac{1}{d^2}\sum_{r=1}^d\sum_{r'=1}^d\left(\sum_{j=0}^{\lfloor mss'/d^2\rfloor}\sum_{\l\in C_s^{d^2/(ss')}(jss'/d^2)}\prod_{i=1}^s{\l_iss'/d^2 + s' - 1\choose s'-1}\right)
    \end{align*}
    where for each $r, r'$ we define $s:= \gcd(r, d)$ and $s':=\gcd(r', d).$
\end{theorem}

\begin{proof}
    Let $\s$ be a $d-$cycle generating $\Z/d\Z.$ Then, the equality
    \begin{align*}
        \dim\R[V]_m^{G} &= \frac{1}{d^2}\sum_{r=1}^d\sum_{r'=1}^d\chi'(\s^r, \s^{r'})
    \end{align*}
    holds by Lemma \ref{lem:character formula}. As in Lemma \ref{lemma: cyclic invariant dimension}, the character $\chi'(\s^r, \s^{r'})$ is equal to the number of monomials in $\R[V]_m$ that are fixed by the pair of permutations $(\s^r, \s^{r'}).$ To each monomial $X^\a$ in $\R[V]_m$ we associate the matrix of exponents $\a$ and $X^\a$ is fixed by $(\s^r, \s^{r'})$ if and only if $\a = M_{\s^r} \a (M_{\s^{r'}})^\T$. So, the number of monomials of degree $j$ that are fixed by $(\s^r, \s^{r'})$ is equal to the number of non-negative matrices whose entries sum to $j$ that are invariant under the action of $(\s^r, \s^{r'}).$

    Let $s = \gcd(r, d)$ and $s'=\gcd(r', d).$ Then, $\s^r = \s_1\ldots\s_s$ decomposes as a product of $s$ disjoint cycles of length $d/s$ and $\s^{r'}$ decomposes as a product of $s'=\t_1\ldots\t_{s'}$ disjoint cycles of length $d/s'.$ In order for a matrix $\a$ to be invariant under $(\s^r, \s^{r'}),$ the rows corresponding to each $\s_i$ must be equal and the columns corresponding to each $\t_i$ must be equal. It follows that the row sums of $\a$ must correspond to a weak composition $\l = (\l_1, \ldots, \l_s)$ of $j$ into $s$ parts divisible by $d/s.$ Then, the rows corresponding to $\s_i$ are all equal, have entries that sum to $\l_i/(d/s)$ and are invariant under column permutations by $\s^{r'}.$ So, each row corresponding to $\s_i$ is a weak composition of $\l_is/d$ into $s'$ parts divisible by $d/s'.$ It follows that $\l_i$ must be divisible by $d^2/(ss')$ and so $\l$ is a weak composition of $j$ into $s$ parts divisible by $d^2/(ss').$ Thus, the total number of degree $j$ monomials fixed by $(\s^r, \s^{r'})$ is $\sum_{\l\in C_s^{d^2/(ss')}(j)}\prod_{i=1}^s C_{s'}^{d/s'}(\l_is/d) = \sum_{\l\in C_s^{d^2/ss'}(j)}\prod_{i=1}^s{\l_iss'/d^2 + s' - 1\choose s'-1}.$ Then, the equality holds:
    \begin{align*}
        \chi'(\s^r, \s^{r'}) &= \sum_{j=0}^m\sum_{\l\in C_s^{d^2/(ss')}(j)}\prod_{i=1}^s{\l_iss'/d^2 + s' - 1\choose s'-1}.
    \end{align*}
    In order for the set $C_s^{d^2/ss'}(j)$ to be non-empty we must have that $j$ is divisible by $d^2/ss'$ and so this is equivalent to:
    \begin{align*}
        \chi'(\s^r, \s^{r'}) &= \sum_{j=0}^{\lfloor mss'/d^2\rfloor}\sum_{\l\in C_s^{d^2/(ss')}(jss'/d^2)}\prod_{i=1}^s{\l_iss'/d^2 + s' - 1\choose s'-1}.
    \end{align*}
    Then, by Lemma \ref{lem:character formula}, we have
    \begin{align*}
        \dim\R[V]_m^G &= \frac{1}{d^2}\sum_{r=1}^d\sum_{r'=1}^d\left(\sum_{j=0}^{\lfloor mss'/d^2\rfloor}\sum_{\l\in C_s^{d^2/(ss')}(jss'/d^2)}\prod_{i=1}^s{\l_iss'/d^2 + s' - 1\choose s'-1}\right)
    \end{align*}
    The proof is complete.
\end{proof}

We now use the equality in Theorem \ref{theorem: image rank} to show that $\dim\R[V]_m^{\Z/d\Z\times\Z/d\Z}$ does not stabilize and is polynomial in $d.$ The right hand side of the equality in Theorem \ref{theorem: image rank} is minimized when $d$ is a prime number and $m$ is less than $d.$ If $d$ is prime and $\s$ generates $\Z/d\Z,$ then for any choice of $r, r'<d,$ the permutations $\s^r$ and $\s^{r'}$ are both $d-$cycles. In order for a monomial $X^\a$ of degree $j$ to be fixed by $(\s^r, \s^{r'})$, all entries in the exponent matrix $\a$ must be equal. However, since we assume that $j<d^2,$ this condition can only be satisfied when $j=0$ and so $\chi'(\s^r, \s^{r'}) = 1.$ If exactly one of $r, r'$ is equal to $d,$ the monomial $X^\a$ is fixed by $(\s^r, \s^{r'})$ is fixed if and only if all rows/columns in $\a$ are equal. Since the degree of the monomial is less than $d,$ this is once again only true when all entries of $\a$ are zero and so $\chi'(\s^r, \s^{r'})=1.$ Thus, $\chi'(\s^r, \s^{r'})>1$ if and only if $r=r'=d$ and so $(\s^r, \s^{r'})$ is the identity. Since every monomial in $\R[V]_m$ is fixed by the identity, $\chi(\s^d, \s^d) = \dim\R[V]_m = \sum_{j=0}^m{j+d^2-1\choose d^2-1}.$ Then, we see that
\begin{align*}
    \dim\R[V]_m^{\Z/d\Z\times\Z/d\Z} &= \frac{1}{d^2}\left(d^2 + \sum_{j=1}^m{j+d^2-1\choose d^2-1}\right) \\
    &> \frac{d^{2m-2}}{m^m}.
\end{align*}
It follows that $\dim\R[V]_m^{\Z/d\Z\times\Z/d\Z}$ is polynomial in $d$ and does not stabilize.

\section{Representation stability and free kernels}\label{sec:stability_background}

We have seen in several interesting and relevant examples that the rank of $G-$invariant kernels becomes dimension-independent because $\dim(\R[V]_m^G)$ is independent of $\dim V$. 
This stabilization can be understood in the context of a broader phenomenon called  \textit{representation stability} \cite{church2013representation, church2014representation}. In this section, we introduce some of the key definitions from the literature on representation stability and show that our examples fit into this framework. 
Importantly, we will be able to parametrize invariant kernels in a way that enables comparing data from different dimensions of interest. As we will see in the next section, this is useful for learning problems where we have access to training data in multiple dimensions and would like to learn a function that can take inputs across dimensions. 

\subsection{Representation stability background.}
The basic objects of study are a sequence of groups $\{G_d\}_{d\in \mathbb{N}}$ acting on a sequence of Euclidean spaces $\{V_d\}_{d\in\mathbb{N}}$, and satisfying certain compatibility conditions. In particular, we assume that $G_d$ acts on $V_d$ by orthogonal transformations. The following is the formal definition; see \eqref{eqn:diag} for the setup. 
\begin{equation}\label{eqn:diag}
    \begin{tikzcd}
    \cdots \arrow[r, shift left=+1, "\gamma_{d-1}"]\arrow[r,leftarrow, shift left=-1, "\gamma_{d-1}^*"'] & V_d \arrow[r, shift left=+1, "\gamma_{d}"]\arrow[r,leftarrow, shift left=-1, "\gamma_{d}^*"'] \arrow[loop above, "G_d"] & V_{d+1} \arrow[r, shift left=+1, "\gamma_{d+1}"]\arrow[r,leftarrow, shift left=-1, "\gamma_{d+1}^*"'] \arrow[loop above, "G_{d+1}"] & V_{d+2} \arrow[r, shift left=+1, "\gamma_{d+2}"]\arrow[r,leftarrow, shift left=-1, "\gamma_{d+2}^*"'] \arrow[loop above, "G_{d+2}"] & \cdots
    \end{tikzcd}
\end{equation}

\begin{definition}[Consistent sequence]
    Let $G_1\subset G_2\subset \ldots$ be a nested sequence of compact groups acting on a sequence of Euclidean spaces $V_1,V_2,\dots$. Suppose moreover, that there exist linear isometries $\gamma_d:V_d\to V_{d+1}$ that are
    $G_d-$equivariant, meaning 
        $$\gamma_d(gx)=g\gamma_d(x)\qquad \forall x\in V_d,~g\in G_d.$$ 
     Then the sequence of pairs $\cV := \{(V_d, \gamma_d)\}_{d\in \mathbb{N}}$ is called a $\{G_d\}$-\textit{consistent sequence}.
\end{definition}

We can identify $V_d$ with $\gamma_d(V_d) \subseteq V_{d+1}$, in which case the orthogonal projection from $V_{d+1}$ onto $V_d$ coincides with the adjoint of $\gamma_d$ \cite[Section 1.1.1]{levin2024any}:  $$\gamma_d^*=\proj_{V_d}\qquad \forall d\in \mathbb{N}.$$ 
We say that a sequence $\{x_d \in V_d\}$ is a \emph{free element} of $\{V_d\}$ if the following holds:
$$x_d = \proj_{V_d}(x_{d+1})\qquad \forall d\in \mathbb{N}.$$ 
Thus, points of a free element in higher dimensions project onto points in lower dimensions.  

Importantly, when the groups $G_d$ are in a sense ``sufficiently large'', the invariant spaces $V_d^{G_d}$ become isomorphic to each other and hence of the same dimension. Let us give a name to this situation.

\begin{definition}[Stability degree]
Given a  consistent sequence $\cV=\{(V_d, \gamma_d)\}_{d\in\N}$, its {\em stability degree} is the smallest  $D\in \N$ such that the maps $\proj_{V_d}$ restrict to an isomorphism between $V_{d+1}^{G_{d+1}}$ and $V_{d}^{G_{d}}$ for all $d\geq D$.
\end{definition}

We will focus on one sufficient condition, from  \cite[Definition 2.2]{levin2024any}, which ensures that the stability degree is finite.
\begin{definition}[Generation degree]
    The consistent sequence $\cV=\{(V_d, \gamma_d)\}_{d\in\N}$ is \emph{finitely generated} if there exists $D \in \N$ such that the equality 
    $$\R[G_d]V_D =V_d\qquad \textrm{holds for all } d\geq D.$$
 The {\em generation degree} of $\cV$ is the smallest such $D.$ 
\end{definition}
Recall that $\R[G_d] = \big\{\sum_{i=1}^r \alpha_i g_i \mid r \in \N, \text{and } \forall i \in [r], \alpha_i \in \R, g_i \in G_d\big\}$ is the set of linear combinations of finitely many elements in $G_d$, while  $\R[G_d] V_D = \left\{\rho \cdot v \mid \rho \in \R[G_d], v \in V_d\right\}$ is the set of all products between $\R[G_d]$ and $V_D.$ Here is a small example. 

\begin{example}\label{ex:generation}
    Take $G_d = \fS_d$ acting by permuting coordinates of $V_d = \R^d$ and let $\gamma_d \colon \R^d \rightarrow \R^{d+1}$ be the zero padding map that sends $x \mapsto (x,0)$. Then,
    the generation degree of $\cV$ is one, since $\R$ is identified with $\text{span}\{e_1\} \subseteq \R^d$ and we can generate any vector $v \in \R^d$ via 
    $
    v = \sum_{i=1}^d v_i g_i \cdot e_1,
    $
    where $g_i$ permutes the first and $i$th coordinates only, i.e., $g_i$ is the transposition $(1i)$. In this example, the stability degree is also one, since $({\R^d})^{\fS_d}$ is the line spanned by the all-ones vector for any $d$.
\end{example}

Importantly, finite generation ensures that the invariants stabilize \cite[Proposition 2.3]{levin2024any}.
\begin{proposition}[Finite generation implies stability] \label{prop:rep-stability}
    If $\cV=\{(V_d, \gamma_d)\}_{d\in\N}$ is a consistent sequence generated in degree $D,$ the orthogonal projections $\proj_{V_d}:V_{d+1}^{G_{d+1}}\to V_d^{G_d}$ are injective for all $d\geq D$ and isomorphisms for large $d.$
\end{proposition}
We refer the reader to \cite[Appendix B]{levin2024any, levin2023free} for a discussion of a possible gap between the stability and generation degrees, although in our main examples they coincide. 

A direct corollary of Proposition~\ref{prop:rep-stability} is the existence of a \emph{free basis} of invariants.

\begin{lemma}[Free basis]
Let $\cV=\{(V_d, \gamma_d)\}_{d\in\N}$ be a consistent sequence with stability degree $D$. Then there exists $k\in \mathbb{N}$ and a collection of free elements $\{b_i^{(d)}, \dots, b_k^{(d)}\}_{d\in \mathbb{N}}$ such that $b_1^{(d)}, \dots b_k^{(d)}$ forms a basis of $V_d^{G_d}$ for all $d\geq D$.  
\end{lemma}

In particular, we can parameterize any free invariant element $\{x_d \in V_d^{G_d}\}$ using a finite amount of parameters $\alpha \in \R^k$ by the expression 
$$x_d = \sum^k_{i} \alpha_i b_i^{(d)}\qquad \forall d\geq D,$$
where $D$ is the stability degree. We emphasize that the coefficients $\alpha_i$ fully describe the free-element $\{x_d \in V_d^{G_d}\}$ and are in particular dimension-independent. We will see that this is the underlying reason why polynomial kernels have constant rank in our examples.

\subsection{Representation stabiliy of polynomial invariants and free kernels.}
We now turn to representation stability of polynomial invariants---the main setting of our paper.
To this end, it is elementary to see that any consistent sequence of $\cV=\{(V_d, \gamma_d)\}_{d\in \mathbb{N}}$ induces a consistent sequence of polynomial spaces  $\cF_m = \{(\R[V_d]_m, \gamma_d)\}$ upon setting $\gamma_d(f) := f \circ \proj_{V_d}$, meaning that if $x \in \R[V_d]_m$, then $\gamma_d(f)((x,x_{d+1}) ) = f(x)$. The sequence \eqref{eqn:diag} induces the consistent sequence depicted in \eqref{eqn:diag_poly}.
\begin{equation}\label{eqn:diag_poly}
    \begin{tikzcd}
    \cdots \arrow[r, shift left=+1, "\gamma_{d-1}"]\arrow[r,leftarrow, shift left=-1, "\gamma_{d-1}^*"'] & \mathbb{R}[V_d]_m \arrow[r, shift left=+1, "\gamma_{d}"]\arrow[r,leftarrow, shift left=-1, "\gamma_{d}^*"'] \arrow[loop above, "G_d"] & \mathbb{R}[V_{d+1}]_m \arrow[r, shift left=+1, "\gamma_{d+1}"]\arrow[r,leftarrow, shift left=-1, "\gamma_{d+1}^*"'] \arrow[loop above, "G_{d+1}"] & \mathbb{R}[V_{d+2}]_m \arrow[r, shift left=+1, "\gamma_{d+2}"]\arrow[r,leftarrow, shift left=-1, "\gamma_{d+2}^*"'] \arrow[loop above, "G_{d+2}"] & \cdots
    \end{tikzcd}
\end{equation}
In particular, When $\cF_m$ has stability degree $D$, for all $d \geq D$, there exists a free basis $\{f_1^{(d)}, \dots f_k^{(d)} \in \R[V_d]_{m}^{G_d}\}$ of invariant polynomials. 
In this case, a free element of $\cF_m$ is any sequence $\{f_d \in \R[V_d]_m\}$ satisfying $f_d(x) = f_{d+1}(\gamma_d(x))$ for all $x \in V_d$ and $d\geq D$. The following simple example illustrates this setup.

\begin{example} \label{ex:linear-monomial-basis}
    Recall the setting of \cref{ex:generation} where $G_d = \fS_d$ and $\{V_d = \R^d\}.$ The zero padding map $\gamma_d(x) = (x, 0)$ makes $\{(\R^d, \gamma_d)\}$ a consistent sequence. Then, the polynomial vector spaces $\R[V_d]_m$ also form a consistent sequence when paired with the transition map $\gamma_d(f)(x_1, \dots, x_d, x_{d+1}) = f(x_1, \dots, x_d)$. In the linear case $m = 1$,  the stability degree of the consistent sequence of polynomial spaces is  $D = 1$ and has the free basis $$\left\{f_1^{(d)}(x) = 1, f_2^{(d)}(x) = \sum^d_{i = 1}x_i \right\}.$$
\end{example}

We now return to the settings in Examples~\ref{thm:Permutations}, \ref{thm:Set-permutations}, \ref{theorem: graphs rank}, and \ref{theorem: point cloud rank} and show that the corresponding sequences $(\mathbb{R}[V_d]_m, \gamma_d)$ are consistent and have finite stability degree. The following result can be derived as a corollary of \cite[Theorem A.13]{levin2023free}; however, its statement and proof require additional category theory notions. We include a direct proof for completeness.

\begin{theorem}[Stability degree in the running examples]
 Let $V_d$ and $G_d$ be the vector spaces and groups appearing in Examples~\ref{thm:Permutations}, \ref{thm:Set-permutations}, \ref{theorem: graphs rank}, and \ref{theorem: point cloud rank}. Then, for the 
 zero-padding map $\gamma_d$, the sequences $\mathcal{V}=\{(V_d, \gamma_d)\}_{d\in \mathbb{N}}$ are consistent. Moreover, the stability degree of $\mathcal{F}_m=\{(\mathbb{R}[V_d]_m, \gamma_d)\}_{d\in \mathbb{N}}$ is bounded by $D=m$, $D=m$, $D=2m$, and $D=2m$, respectively. 
\end{theorem}

\begin{proof} %
We first prove that each sequence $\cV = \{(V_d, \gamma_d)\}_{d\in\N}$ is consistent. By definition, $G_1\subset G_2\subset \ldots$ are sequences of compact groups acting linearly on the Euclidean spaces $\{V_d\}.$ The zero-padding map $\gamma_d$ is a linear isometry, and thus we need only check that $\gamma_d$ is $G_d-$equivariant. In each of the sequences, the zero-padding map $\gamma_d$ is equal to $\gamma_d = \id_d\oplus 0$ where $\id_d$ is the identity map on $V_d.$ Moreover, as the group $G_d$ is a subgroup of $G_{d+1},$ the action of $G_d$ on $V_{d+1}$ is given by $g\cdot x = (g\oplus\id_1)(x)$ for all $g\in G_d$ and $x\in V_{d+1}.$ Then, for any $x\in V_d,$ the equalities $\gamma_d(g\cdot x) = g\cdot x\oplus 0 = (g\oplus\id_1)(x\oplus 0) = g\cdot \gamma_d(x).$ It follows that $\gamma_d$ is $G_d-$equivariant. 

Now, we show that each consistent sequence $\cF_m = \{(\R[V_d]_m, \gamma_d)\}_{d\in\N}$ is generated in dimension $D,$ i.e. for all $d\geq D,$ the equality $\R[V_d]_m = \R[\fS_d]\R[V_D]_m.$ %
Note that if any consistent sequence is generated in dimension $D,$ it is also generated in dimension $d\geq D$ and thus we need only show that $\R[V_{D+1}]_m = \R[\fS_{D+1}]\R[V_D]_m.$ Since any polynomial is a linear combination of monomials, it suffices to show that for each monomial $f_{D+1}\in\R[V_{D+1}]_m$ there exists a group element $g\in G_{D+1}$ such that $f_{D+1} = g\cdot f_D$ for some monomial $f_D\in\R[V_D]_m.$

First, consider the sequence $\{(V_d, \gamma_d)\}_{d\in\N}$ appearing in Example~\ref{thm:Permutations}. Let $f_{m+1} = \prod_{i=1}^{m+1}x_i^{a_i}$ be a monomial in $\R[V_{m+1}]_m.$ If $a_{m+1} = 0,$ then $f_{m+1}\in\R[V_m]_m$ so assume that $a_{m+1}>0.$ If $a_{m+1}>0,$ there exists $j\leq m$ such that $a_j = 0.$ Define the monomial $f_m = \left(\prod_{i=1}^mx_{i}^{a_i}\right)\cdot x_j^{a_{m+1}}\in\R[V_m]_m$ and let $\s\in\fS_{m+1}$ be the transposition defined by $\s(j) = m+1$ and $\s(m+1) = j.$ Then, $f_{m+1} = \s\cdot f_m$.

Now, consider the sequence $\{(V_d, \gamma_d)\}_{d\in\N}$ appearing in Example~\ref{thm:Set-permutations}. Recall that data points in $V_d$ are matrices $X\in\R^{d\times k}.$ Let $f_{m+1} = \prod_{i=1}^{m+1}X_i^{A_i}$ be a monomial in $\R[V_{m+1}]_m,$ where each $X_i$ is the vector $[x_{i1}, \ldots, x_{ik}]$ and $A_i$ is the vector $[a_{i1}, \ldots, a_{ik}]$ of exponents. Then $\sum a_{ij} = m$. As in the previous example, if $A_{m+1}$ is the zero vector,  $f_{m+1}\in\R[V_m]_m.$ Suppose that $A_{m+1}$ is not the zero vector. Then, there exists some $j\leq m$ such that $A_j$ is the zero vector. Define the monomial $f_m = \left(\prod_{i=1}^m X_i^{A_i}\right)\cdot X_j^{A_{m+1}} \in\R[V_m]_m$ and let $\s\in\fS_{m+1}$ be the transposition defined by $\s(j) = m+1$ and $\s(m+1) = j.$ Then, the equality $f_{m+1} = \s\cdot f_m$ holds.

Consider the sequence $\{(V_d, \gamma_d)\}_{d\in\N}$ appearing in Example~\ref{theorem: graphs rank}. Let $f_{2m+1} = \prod_{i=1}^{2m+1}X_i^{A_i}$ be a monomial in $\R[V_{2m+1}]_m.$ Since elements of $V_d$ are symmetric matrices, we assume without loss of generality that the matrix $A$ with rows $A_i$ is also symmetric. If $A_{2m+1}$ is the zero vector, we have that $f_{m+1}\in\R[V_{2m}]_m.$ Suppose that $A_{2m+1}$ is not the zero vector. Then, we have that there exists $j\leq 2m$ such that $A_j$ is the zero vector. As in the previous arguments, define the monomial $f_{2m} = \left(\prod_{i=1}^{2m} X_i^{A_i}\right)\cdot X_j^{A_{2m+1}} \in\R[V_{2m}]_m$ and let $\s\in\fS_{2m+1}$ be the transposition defined by $\s(j) = 2m+1$ and $\s(2m+1) = j.$ Then, the equality $f_{2m+1} = \s\cdot f_{2m}$ holds.

Finally, consider the sequence $\{(V_d, \gamma_d)\}_{d\in\N}$ appearing in Example~\ref{theorem: point cloud rank}. Recall that the data points are matrices $X\in\R^{d\times k}$ and the group $G_{d}$ consists of both simultaneous orthogonal transformations on the rows of $X$ and permutations of the rows of $X.$ Let $f_{2m+1} = \prod_{i=1}^{2m+1}X_i^{A_i}$ be a monomial in $\R[V_{2m+1}]_m,$ where each $X_i$ is the vector $[x_{i1}, \ldots, x_{ik}]$ and $A_i$ is the vector $[a_{i1}, \ldots, a_{ik}]$  of exponents. As in the previous example, if $A_{2m+1}$ is the zero vector, $f_{2m+1}\in\R[V_{2m}]_m.$ Suppose that $A_{2m+1}$ is not the zero vector. Then, there exists some $j\leq 2m$ such that $A_j$ is the zero vector. Define the monomial $f_{2m} = \left(\prod_{i=1}^{2m} X_i^{A_i}\right)\cdot X_j^{A_{2m+1}} \in\R[V_{2m}]_m$ and let $\s\in G_{m+1}$ be the transposition defined by $\s(j) = m+1$ and $\s(m+1) = j.$ Then, the equality $f_{2m+1} = \s\cdot f_{2m}.$ holds.

Now, since each of the sequences are generated in dimension $D=m, D=m, D=2m,$ and $D=2m$ respectively, by Proposition \ref{prop:rep-stability}, the orthogonal projections $\proj_{\R[V_d]_m}:\R[V_{d+1}]_m^{G_{d+1}}\to \R[V_d]_m^{G_d}$ are injective for all $d\geq D.$ However, by Theorems \ref{thm:Permutations}, \ref{thm:Set-permutations}, \ref{theorem: graphs rank}, and \ref{theorem: point cloud rank}, we have that the dimensions $\dim \R[V_{d}]_m^{G_d}$ are constant for $d\geq D$ and so the orthogonal projections $\proj_{\R[V_d]_m}$ are isomorphisms. It follows that the stability degree of each sequence $\cF_m$ is at most $D$ as desired.
\end{proof}

Typical kernels, including those that are polynomials in the inner product, can be used to compare data across different dimensions. For example, the polynomial kernel $K(x,y)=(1+x^\top y)^m$ can be used to compare data points $x\in \R^{d_1}$ and $y\in\R^{d_2}$ with $d_1\leq d_2$ by simply zero-padding $x$ into a vector in $\R^{d_2}$.
Formally, we define a \textit{free kernel} as follows.

\begin{definition}[Free Kernel]
    Let $\mathcal{K} = \{K_d:V_d\times V_d\to\R\}_{d\in\N}$ be a sequence of kernels and let $\mathcal{V}=\{(V_d, \gamma_d)\}_{d\in\N}$ be a consistent sequence. We say that $\mathcal{K}$ is a \textit{free kernel} if for any dimension $d$ and any choice of $x,y\in V_d,$ the equality $K_d(x, y) = K_{d+1}(\gamma_d(x), \gamma_d(y))$ holds. Moreover, $\mathcal{K}$ is a free $G-$invariant kernel if in addition, each $K_d$ is $G_d-$invariant.
\end{definition}

Let $\mathcal{K}$ be a free kernel. Recall that Lemma~\ref{lemma: C is unique} showed that for each $d$, we may write 
$$K_d(x,y)=F_d(x)^\top C_d F_d(y),$$
where $F_d=(f^{(d)}_1,\ldots,f^{(d)}_k)^\top$ stacks some basis for $\mathbb{R}[V_d]_m$, the integer $k$ is bounded by the rank of $K_d$, and $C_d\in \mathbb{R}^{k\times k}$ is a symmetric matrix. Now if $\{(\R[V_d]_m,\gamma_d)\}_{d\in \mathbb{N}}$ is a consistent sequence with finite stability degree $D$, then by switching basis, we may ensure that $f^{(d)}_1,\ldots,f^{(d)}_k$ is a free basis of $\mathbb{R}[V_d]_m$. The following theorem shows that upon this change of basis the matrix $C_d$ becomes independent of $d$. In words, this means that free kernels admit a free description.

\begin{lemma}[Free kernels admit a free parametrization]\label{lem:uniqueC}
    Suppose that $\mathcal{V}=\{(V_d,\gamma_d)\}$ is a consistent sequence and that the induced consistent sequence $\mathcal{F}_m=\{(\R[V_d]_m,\gamma_d)\}$ has stability degree $D$. Let $\mathcal{K}=\{K_d\}_{d\in \mathbb{N}}$ be any free $G-$invariant polynomial kernel of degree at most $m.$ Then, 
    for any free basis $\F_d = (f_1^{(d)}, \dots, f_k^{(d)})^\top$  for $\mathcal{F}_m$,
    there exists a symmetric matrix $C \in \R^{k \times k}$ such that for all $d \geq D,$ we can write 
    \begin{equation}
        K_d(x, y) = \F_d(x)^\top C \F_d(y) \quad \text{for all }x, y \in V_d.
    \end{equation}
\end{lemma}

\begin{proof}
    By Lemma~\ref{lemma: C is unique}, for each $d$ there exist some symmetric matrices $C_d \in \R^{k \times k}$ such that $K_d(x,y) = \F_d(x)^\top C_d \F_d(y)$ for all $x , y \in V_d$, and $\F_d = (f_1^{(d)}, \dots, f_k^{(d)})^\top$ stacks a basis for $\R[V_d]_m.$ Upon switching basis, we may assume that $f_1^{(d)}, \dots, f_k^{(d)}$ is a free basis.
    Next we show that the matrices $C_d$ are equal for all $d\geq D$.   %
    Since $\mathcal{K}$ is a free kernel, for all choices of $x, y$ in $V_d$ the equality $K_d(x, y) = K_{d+1}(\gamma_d(x), \gamma_d(y))$ holds. Thus, we conclude
    \begin{equation*}
        \F_d(x)^\T C_d \F_d(y)= K_d(x, y)  = \F_{d+1}(\gamma_d(x))^\T C_{d+1} \F_{d+1}(\gamma_d(y)).
    \end{equation*}
    Moreover, due to the freeness of the polynomial basis $\{F_d\}_{d\in \mathbb{N}}$,  $\F_{d+1}(\gamma_d(x)) = \F_d(x)$ for all $x\in V_d,$ and therefore 
    \begin{align*}
        \F_{d+1}(\gamma_d(x))^\T C_{d+1} \F_{d+1}(\gamma_d(y)) &= \F_{d}(x)^\T C_{d+1} \F_{d}(y).
    \end{align*}
    It follows that $C_{d+1} = C_d=\ldots=C_D$ for all $d\geq D$, as claimed. 
\end{proof}

\section{Generalization of polynomial regression across dimensions}
\label{sec:generalization}

Standard regression problems seek to determine a function $p^*$ on a vector space $V$ from finitely many noisy observations of the form 
$y_i=p^*(x_i)+\varepsilon_i$. In this section, we consider the problem of learning invariant polynomials from data generated {\em across different dimensions}.
Setting the stage, consider a sequence of groups $\cG = \{G_d\}_{d\in \mathbb{N}}$ acting on a sequence of vector spaces $\{V_d\}_{d\in \mathbb{N}}$.
Suppose that there exists a sequence of (latent) invariant polynomials
$$p_d^{\star} \in \R[V_d]_m^{G_d}\qquad\forall d\geq 1,$$
which generate the observation model
$$y = p_d^*(x) + \varepsilon\qquad \textrm{for}~x\in \R^d.$$
We are provided a training sample $S = \bigcup_{d=1}^K S_{d}$ from this model, where each set $S_d=\{(x_i, y_i)\} \subseteq V_d \times \R$ consists of feature/label pairs in a fixed dimension and the noise values $\varepsilon_i$ are independent and identically distributed with $\EE[\varepsilon]=0$ and $\EE[\varepsilon^2]=\sigma^2$. We will focus on the fixed design setting wherein the feature vectors $x_i$ are fixed, i.e. not random.
 We let $n = |S|$ be the total number of training examples and $X = \{x\}_{(x, y) \in S}$ be the collection of feature vectors (in potentially different dimensions) of the training set.
Our goal is to predict $p^*_{\bar d}$  from $S$ for some arbitrary dimension $\bar d$. Clearly, without making further assumptions, this goal is untenable since the invariants $p_d^{\star}$ may bear no resemblance to $p_{\bar d}^{\star}$ for any $d\neq \bar d$.
 
 In order to make the problem tractable, we assume that $\mathcal{V}=\{(V_d, \gamma_d)\}$ is a consistent sequence relative to $\cG$ for some transition maps $\gamma_d$ and that the sequence of invariant spaces $\cF_m = \{(\R[V_d]^{G_d}_{m},\gamma_d)\}_{d\in \mathbb{N}}$ has a finite stability degree $D<\infty$. Without loss of generality, we will assume $D=1$, since we can always truncate the sequence to begin at $d=D$. In particular, $\cF_m$ admits a free basis $\{\varphi_1^{(d)}, \dots \varphi_k^{(d)} \in \R[V_d]_{m}^{G_d}\}_{d\in \mathbb{N}}$ of invariant polynomials. Thus for all $d\geq 1$ the following two key properties hold:
    \begin{enumerate}
        \item (\textbf{Basis}) The collection $\varphi_1^{(d)}, \dots, \varphi_k^{(d)}$ forms a basis for $\cF_{m}^d = \R[V_d]_{m}^{G_d}$.
        \item (\textbf{Consistency}) Equality holds:
        $$\varphi_i^{(d)}(x)=\varphi_i^{(d+1)}(\gamma_d(x))\qquad \forall x\in V_d,~i \in \{1, \dots, k\}.$$
    \end{enumerate}
In order to link the invariants $p^*_d$ across dimensions, we will assume that $\{p^*_d\}_{d\in \mathbb{N}}$ is a free element of $\mathcal{F}_m$. Thus, we may write 
$$p^*_d=\sum_{i=1}^k \alpha_i^* \varphi_i^{(d)}\qquad \forall d\geq 1$$
for some vector $\alpha^*\in \R^k$.
Henceforth, with a slight abuse of notation, we will drop the dimension $d$ index and simply write $\varphi(x) = (\varphi_1(x), \ldots, \varphi_k(x))^\top$ as the dimension index can be inferred from the dimension of $x$. 
Define now the matrix %
 \begin{equation}\label{eq:training-gram}
\Phi(X) = (\varphi(x_1), \dots, \varphi(x_n)) \in \R^{n \times k}. \end{equation}
We assume throughout that $\Phi(X)$ has rank $k$, which in particular entails $n\geq k$. %

With this notation, we may write $p^*_d(x)=\langle \varphi_d(x),\alpha^*\rangle$ for each $d\geq 1$ and $x\in V_d$. Therefore, we aim to estimate the free element $\{p_d^*\}_{d\in \mathbb{N}}$  by solving the regression problem: %
\begin{equation}\label{eqn:basic_soln_reg}
\mathop{\min}_{\alpha \in \R^k} \frac{1}{2n} \sum_{i = 1}^n \left(\dotp{\varphi(x_i), \alpha} - y_i \right)^2.
\end{equation}
Note that the solution of \eqref{eqn:basic_soln_reg} is simply:
\begin{equation} \label{eq:GOAT}
    \widehat \alpha = \left(\Phi(X)^\top \Phi(X)\right)^{-1}\Phi(X)^\top y.
\end{equation}
The choice $\hat \alpha$ of coefficients induces a sequence of invariant polynomials 
$$\widehat p_d(x) := \dotp{\varphi_d(x), \alpha}\qquad \forall~d\in \mathbb{N}.$$ %
Once more, we will drop the dimension $d$ index for simplicity. We evaluate the performance of this sequence of polynomials using finite test data in a potentially different dimension $T := 
\left\{\left(x_i, y_i\right)\right\} \subseteq \bigcup_{d=1}^\infty V_{d} \times \R $ where once more $y_i = p^{\star}(x_i) + \varepsilon_i$ with i.i.d. noise $\varepsilon_i.$
Given any estimator $\widehat \alpha,$ define the risk
$$R_{T}^{(\alpha^{\star})}(\widehat \alpha)  := \frac{1}{|T|}  \sum_{(x, y) \in T} \EE \left( \widehat p(x) - y  \right)^2,$$
where the index $\alpha^{\star}$ indicates the latent sequence of polynomial generating the labels $y$.
Our goal is to establish a bound on the excess risk, 
 $R_{T}(\widehat \alpha) - \mathop{\min}_{\alpha \in \R^k} R_T(\alpha).$
To this end, let us introduce a notion of Gram matrices that will play a crucial role: for a finite set $\Delta \subseteq \bigcup_{d=1}^\infty V_d$, define the Gram matrix 
$$\Sigma_{\Delta} = \frac{1}{|\Delta|} \sum_{(x, y) \in \Delta} \varphi(x)\varphi(x)^\top.$$
All proofs of results in this section appear in Appendix~\ref{app:generalization}.
\begin{theorem}[\textbf{Upper bound}]\label{thm:bound_estimator}
 The estimator $\widehat \alpha$ in \eqref{eq:GOAT} is unbiased with excess risk:
    \begin{equation}
        R_{T}^{(\alpha^{\star})}(\widehat \alpha) - \mathop{\min}_{\alpha \in \R^k} R_{T}^{(\alpha^{\star})}(\alpha) = \frac{\sigma^2}{n} \Tr\left(\Sigma_S^{-1}\Sigma_T\right).
    \end{equation}
\end{theorem}

\begin{example}\label{ex:instantiation-minimax-risk}
    Recall the setting of~\cref{ex:linear-monomial-basis}. In this case, the basis size is two. For concreteness, assume that the training data $S \subseteq \R^{d} \times \R$ lie in some dimension $d$, while the test set $T \subseteq \R^{\bar d} \times \R$ is independent of $S$ and lives in a different dimension $\bar d.$ Suppose that the feature vectors $\{x\}_{(x, y) \in S}$ in $S$ are drawn i.i.d. from a standard Gaussian distribution $N(0, I_d)$ and similarly the feature vectors in $T$ are i.i.d. with distribution $N(0, I_{\bar d}).$ A quick computation reveals that 
    $$\EE \Sigma_{S} = \begin{pmatrix}
        1 & 0 \\ 0 & d
    \end{pmatrix} \quad \text{and} \quad \EE \Sigma_{T} = \begin{pmatrix}
        1 & 0 \\ 0 & \bar d
    \end{pmatrix}.$$
    Then, for any sufficiently small $\Delta > 0$  the excess risk is bounded by 
    $$\frac{\sigma^2}{n} \Tr\left(\Sigma_S^{-1}\Sigma_T\right) \leq (1+\Delta) \cdot \frac{\sigma^2}{n} \cdot  \left( 1 + \frac{\bar d}{d}\right)$$
    with probability at least $1 - 4\left(\exp(-58 \Delta |S|) + \exp(-58 \Delta |T|)\right)$
    We defer the proof of this bound to Appendix~\ref{app:proof-instantiation-minimax-risk}. When $\bar d = d$ the upper bound scales like the intrinsic dimension of the parameter space over the number of data points, which is standard for regression in a fixed dimension \cite{bach2024learning}. On the other hand, when $\bar d > d,$ the bound reveals a graceful degradation of the excess risk.
\end{example}

In general, the bound in Theorem~\ref{thm:bound_estimator} appears difficult to compute explicitly. Nonetheless, we now show that this bound is optimal among any estimation procedure in a minimax sense.
Namely, as before, suppose we have a fixed samples of feature vectors for training $X_S \subseteq \cup_{d=1}^\infty V_d$ and for testing $ X_T \subseteq \cup_{d=1}^\infty V_d$, both of which might have vectors in different dimensions. 
Given any parameter $\alpha \in \R^{k}$ we consider the dataset $S = \{(x, y) \mid x \in X_S\}$ where for each $x$ the label is given by $y = \dotp{\varphi(x), \alpha} + \varepsilon$, here the noise $\varepsilon$ is i.i.d. with distribution $N(0, \sigma^2).$ We define the test set $T$ in a completely analogous fashion. %

\begin{theorem}[\textbf{Lower bound}]\label{thm:loweround} Assume that $\widehat \alpha$ is any estimator having access to training data $S$ and $\alpha \in \R^k$ is arbitrary. Then, the minimax excess risk on test set $T$ is exactly 
\begin{equation}\label{eq:minimax}
    \inf_{\widehat \alpha} \mathop{\sup}_{\alpha} \left\{{R}_T^{(\alpha)}(\widehat \alpha) - \mathop{\min}_{\alpha' \in \R^k} {R}_T^{(\alpha)}( \alpha') \right\} = \frac{\sigma^2}{|S|} \Tr\left(\Sigma_{S}^{-1} \Sigma_{T} \right).
\end{equation}
\end{theorem}

\section{Computing Free Bases}\label{sec:bases}

This section studies how to compute a free basis of invariant polynomials and provides computationally efficient bases for permutation and set-permutation invariant polynomials. %

Setting the stage, let $\{(V_d, \gamma_d)\}_{d\in\N}$ be a consistent sequence of $\{G_d\}_{d\in\N}-$representations with $\R[V_d]_m$ having finite stability degree. Monomials yield a natural choice for a basis of the polynomial spaces $\R[V_d]_m$; %
the next result shows that if $\gamma_d$ is the zero-padding 
 map, and $G_d\subset\fS_d$ acts by permutation on the basis of $V_d,$ %
 then the set of polynomials obtained by summing all monomials in a $G_d$-orbit yields a free basis of invariants.

\begin{proposition}[Free basis via symmetrization]\label{lemma: monomial free basis}
    Consider $\{(V_d, \gamma_d)\}_{d\in\N}$, a consistent sequence of $\{G_d\}_{d\in\N}-$representations, where $\gamma_d$ is the zero-padding map, $G_d\subset\fS_d$ acts by permuting coordinates of $V_d$, 
    and suppose $\cF_m=\{(\R[V_d]_m, \gamma_d)\}_{d\in\N}$ has stability degree $D.$ Then, if $f_1^{(d)}, \ldots, f_{N_d}^{(d)}$ is the monomial basis of $\R[V_d]_m,$ the set $$B_d = \left\{\frac{1}{c_i}\sum_{g\in G_d}g\cdot f^{(d)}_i:i=1, \ldots, N_d\right\}_{d\in\N}$$ of monic polynomials forms a free basis of invariants for $\cF_m,$ where each $c_i$ is a normalizing constant. \footnote{ The normalizing constants $c_i$ can be explicitly computed by taking the $\fS_d$-sum of a monomial and dividing so that the result is monic.}
\end{proposition}

\begin{proof}
    By Lemma~\ref{lemma: rank = number of orbits}, we have that $B_d$ forms a basis of $\R[V_d]_m^{\fS_d}$ for each dimension $d\geq D$ and so it remains to check that each $G_d-$summed monomial is a free element of $\{\R[V_d]_m\}_{d\in\N}.$ Picking a monomial $f_i^{(d+1)}$, each monomial $g\cdot f_i^{(d+1)}$ appears $c_i$ times in the sum $\sum_{g\in G_{d+1}}g\cdot f_i^{(d+1)}.$ By dividing by the normalizing constant $c_i,$ we ensure that the coefficient on each monomial in the orbit is one. Since $\gamma_d$ is the zero-padding map, the projection $\gamma_d^*$ sets the value of variables that are in $\R[V_{d+1}]_m$ but not $\R[V_d]_m$ equal to zero. It follows that $\gamma_d^*\left(\frac{1}{c_i}\sum_{g\in G_{d+1}}g\cdot f^{(d+1)}_i\right)$ is the sum of all monomials in the orbit of $f_i^{(d+1)}$ that are also contained in $\R[V_d]_m.$ %
    If two monomials $f_{j}^{(d+1)}, f_{j'}^{(d+1)}$ are in the orbit of $f_i^{(d+1)}$ and are contained in $\R[V_d]_m,$ there exists some permutation $\s\in G_d$ such that $f_j^{(d+1)} = \s f_{j'}^{(d+1)}.$ 
    Setting $f_j^{(d)} := \gamma_d^*(f_j^{(d+1)})$, it follows that  $$\gamma_d^*\left(\frac{1}{c_i}\sum_{g\in G_{d+1}}g\cdot f^{(d+1)}_i\right)= \frac{1}{c_j}\sum_{g\in G_d}g\cdot f^{(d)}_j$$ holds for some $j$ and thus $\frac{1}{c_i}\sum_{g\in G_d}g\cdot f_i^{(d+1)}$ is a free element as desired.
\end{proof}

In particular, taking $V_d$ and $G_d=\fS_d$ to be the vector spaces and groups appearing in Examples \ref{thm:Permutations}, \ref{thm:Set-permutations}, \ref{theorem: graphs rank}, and \ref{theorem: point cloud rank}, we can apply Lemma \ref{lemma: monomial free basis} to conclude that the $\fS_d$-averaged monomials form a free basis of invariants for $\R[V_d]_m^{\fS_d}.$ However, as the dimension $d$ increases, calculating the $\fS_d$-orbit of monomials becomes prohibitively costly. 
Thus, it is desirable to design free bases of invariants that remain tractable as $d$ grows.
While we are not aware of such constructions for graphs (Example~\ref{theorem: graphs rank}) and point clouds (Example~\ref{theorem: point cloud rank}), we provide efficiently-computable free bases for permutation invariant (Example~\ref{thm:Permutations}) and set permutation invariant kernels (Example \ref{thm:Set-permutations}) in the following two sections, respectively. 

\subsection{Permutation Invariants}\label{sec:permutation-basis}
The foundation for the free bases of both permutation and set permutation invariant polynomials are \textit{elementary symmetric polynomials}. 

\begin{definition}[Elementary Symmetric Polynomials]
    The \textit{elementary symmetric polynomial}  of degree $k$ in $d$ variables $x=(x_1, \ldots, x_d)$ is  $e_k(x) = \sum_{1\leq i_1<\ldots<i_k\leq d}\prod_{j=1}^{k}x_{i_j}.$
    For a given partition $\l\vdash m,$ \textit{the elementary symmetric polynomial indexed by $\l$} is $e_\l(x) = e_{\l_1}(x)\ldots e_{\l_d}(x)$ where $e_{\l_i}(x)$ is the $\l_i^{th}$ elementary symmetric polynomial.
\end{definition}

Elementary symmetric polynomials are well studied combinatorial objects. In particular,  they form an algebra basis of the ring of symmetric polynomials, and the elementary symmetric polynomials indexed by partitions form a vector space basis of the space of symmetric polynomials \cite{fleischmann-invariantTheory,smith-InvariantTheory,stanley-ec2,   weyl1946classical}. 
The key to the efficient computability of elementary symmetric polynomials is the following recursion:
\begin{equation}\label{eq:esp-recursion}
    e_{j}(x_1, \ldots, x_d) = x_d e_{j-1}(x_1, \ldots, x_{d-1}) + e_j(x_1, \ldots, x_{d-1})
\end{equation}
We leverage this recursion to show that the elementary symmetric polynomials indexed by partitions form a free basis of invariants.

\begin{proposition}[Free basis of permutation invariant polynomials]
    Let $V_d = \R^d$ and consider the consistent sequence of $\{\fS_d\}_{d\in\N}$-representations $\cF_m = \{(\R[V_d]_m, \gamma_d)\}_{d\in \N}$ where $\gamma_d$ is the zero-padding map. The set $\{e_\l(x_1, \ldots, x_d):\l\vdash j\leq m\}$ of all elementary symmetric polynomials indexed by partitions forms a free basis of invariants for $\cF_m.$
\end{proposition}
\begin{proof}
    Elementary symmetric polynomials indexed by partitions
    of $j \leq m$ form a vector space basis of $\R[V_d]_m^{\fS_d}$ and so we need only prove that they are free elements. 
    Since $\gamma_d$ is the zero-padding map, for each $f(x_1, \ldots, x_d)\in\R[V_d]_m$ we have that $\gamma_d^*f(x_1, \ldots, x_d) = f(x_1, \ldots, x_{d-1}, 0).$ So, by \eqref{eq:esp-recursion}, for each elementary symmetric polynomial $e_j(x_1, \ldots, x_d)$, \begin{align*} 
    \gamma_d^*(e_j(x_1, \ldots, x_d)) & = e_j(x_1, \ldots, x_{d-1},0) 
    \\ & = e_j(x_1, \ldots, x_{d-1}) + 0\cdot e_{j-1}(x_1, \ldots, x_{d-1}) = e_j(x_1, \ldots, x_{d-1}).
    \end{align*}
    It follows that the elementary symmetric polynomials are free elements of $\cF_m.$ Moreover, for a given partition $\l,$ we can see that
    \begin{align*}
        \gamma_d^*(e_\l(x_1, \ldots, x_d)) &= \gamma_d^*(e_{\l_1}(x_1, \ldots, x_d)\cdots e_{\l_d}(x_1, \ldots x_d)) \\
        &= e_{\l_1}(x_1, \ldots, x_{d-1}, 0)\cdots e_{\l_d}(x_1, \ldots x_{d-1}, 0) \\
        &= \prod_{i=1}^d(e_{\l_i}(x_1, \ldots, x_{d-1}) + 0\cdot e_{\l_i-1}(x_1, \ldots, x_{d-1}))\\
        &=e_{\l}(x_1, \ldots, x_{d-1})
    \end{align*}
    and so the elementary symmetric polynomials indexed by partitions of $j \leq m$ are also free elements of $\cF_m.$ 
\end{proof}

For any given dimension $d$, and $r = \dim(\R[V_d]_m^{\fS_d})$, let $E_m^{(d)} \colon \R^d \rightarrow \R^r$ %
be a mapping producing a vector of all evaluations of elementary symmetric polynomials indexed by partitions of degree at most $m.$ As before we will drop the index $d$ and simply write $E_m.$  Suppose that one has access to the matrix $C$ representing a kernel $K$ with respect to the basis of elementary symmetric polynomials indexed by partitions. Then, for data points $x, y\in\R^d,$ Procedure \ref{alg:permutation-invariant} describes the evaluation of $K(x, y).$

\begin{algorithm}
\caption{Permutation Invariant Kernel Computation}\label{alg:permutation-invariant}
\begin{algorithmic}[1]
    \REQUIRE Data points $x, y\in \R^d,$ matrix $C$ representing $K.$
    \STATE Set $E_m(x) = E_m(y) = 0 \in \R^r$
    \FORALL{$j$ such that $0\leq j\leq m$}\label{step:esp}
    \STATE Store $e_j(x)$
    \STATE Store $e_j(y)$
    \ENDFOR
    \FOR{$j=0$ to $m$}\label{step:esp-partition}
    \FORALL{$\l\vdash j$}
    \STATE Set $E_m(x)_\l = \prod_{i=1}^de_{\l_i}(x)$ 
    \STATE Set $E_m(y)_\l = \prod_{i=1}^de_{\l_i}(y)$
    \ENDFOR
    \ENDFOR
    \ENSURE $E_m(x)^\T C E_m(y)$
\end{algorithmic}
\end{algorithm}

\begin{proposition}[Complexity of evaluating a permutation invariant kernel]\label{prop:permutation-complexity} Suppose that $\{K_d\colon V_{d} \times V_d \rightarrow \R\}$ is a free kernel with rank equal to $r = \dim(\R[V_d]_m^{\fS_d}) $. Suppose one has access to the matrix $C \in \R^{r \times r}$ representing $K_d$ in the basis of elementary symmetric polynomials indexed by partitions, i.e., $K_d(x, y) = E_m(x)^\top C E_m(y)$. Then, the number of floating point operations required to compute $K_d(x ,y)$ following Procedure \ref{alg:permutation-invariant} is
$$4md + \sum_{j=0}^m\sum_{\l\vdash j}(|\l|-1) + \left(\sum_{j=0}^mp(j) +1 \right)\left(2\sum_{j=0}^mp(j) - 1\right).
$$
\end{proposition}
\begin{proof}
    The number of flops necessary to compute $K_d(x, y)$ can be upper bounded as follows.
\begin{enumerate}
\vspace{-10pt}
    \item The number of flops required to complete the loop in Step \ref{step:esp} is given by the number of flops necessary to compute the elementary symmetric polynomials in the entries of $x$ and $y.$ Thanks to \eqref{eq:esp-recursion}, computing the elementary symmetric polynomials requires filling in the entries of an $m\times d$ matrix, with each entry requiring one addition and one multiplication. So, in total, evaluating the elementary symmetric polynomials in $x$ and $y$ takes $4md$ flops.
    \item For each partition $\l,$ combining the elementary symmetric polynomials to form the elementary symmetric polynomial indexed by $\l$ takes $|\l|-1$ multiplications, where $|\l|$ is the number of parts of $\l.$ Thus, the number of flops required to complete the loop in Step \ref{step:esp-partition} by generating all elementary symmetric polynomials indexed by partitions of all $j \leq m$ is $\sum_{j=0}^m\sum_{\l\vdash j}\left(|\l|-1\right)$ operations. 
    \item Finally, evaluating the output $E_m(x)^\T CE_m(y)$ via naive matrix multiplication takes $(\sum_{j=0}^m p(j) +1)(2\sum_{j=0}^m p(j) - 1)$ operations. 
\end{enumerate}
\vspace{-10pt} The sum of these operations amounts to the stated bound.
\end{proof}
 One can further upper bound the number of flops by by $4md + (m^2 + m)p(m) +2m^2p(m)^2$, and subsequently using the Hardy-Ramanujan asymptotic formula for $p(m)$, the number of flops grows at a rate of $4md + (m+1)\exp\left(\pi\sqrt{{2m}/{3}}\right) + \exp\left(2\pi\sqrt{{2m}/{3}}\right).$ Note that this completes the proof of the first statement in Theorem \ref{thm:complexity}. While there is an unavoidable exponential dependence on the degree $m,$ there is only a linear dependence on the dimension of the data and thus this method scales well as the degree is held constant and dimension increases.

\begin{example}[Permutation Invariant Taylor Features]\label{ex: invariant tf} We now turn our attention to a concrete invariant kernel that arises in practice and provide an efficient method for its computation. Consider the non-invariant kernel $K(x, y) = \sum_{j=0}^m\frac{\langle x, y\rangle^j}{j!},$ which is a rescaled Taylor features kernel. The permutation invariant Taylor features kernel is
\begin{align*}
    K^*(x, y) &= \sum_{\s\in\fS_d}\sum_{j=0}^m\frac{\langle x,\s(y)\rangle^j}{j!}.
\end{align*}

We provide an efficient way to evaluate $K^*(x, y)$ by explicitly determining the matrix $C$ representing the symmetric bilinear form corresponding to $K^*$ with respect to the basis of elementary symmetric polynomials indexed by partitions. First, we write the kernel as a linear combination of \textit{monomial symmetric polynomials}, a well-known vector space basis of the space of symmetric polynomials. Given a partition $\l\vdash j,$ the monomial symmetric polynomial indexed by $\l$ is %
$m_\l(x)=\frac{1}{\eta(\l)}\sum_{\s\in\fS_d}x^{\s(\l)}$ where $\eta(\l)$ is the product of the factorials of the multiplicities of the parts of $\l.$ 

For a given degree $m$ we successively compute
\begin{align*}
    \sum_{\s\in\fS_d}\langle x, \s(y)\rangle^m &= \sum_{\s\in\fS_d}(x_1y_{\s(1)}+\ldots+x_dy_{\s(d)})^m \\
    &=\sum_{\s\in\fS_d}\sum_{\l\in C_d(m)}{m\choose \l} \prod_{i=1}^d x_i^{\l_i}y_{\s(i)}^{\l_i}\\
    &=\sum_{\l\in C_d(m)}{m\choose\l}\left(\prod_{i=1}^d x_{i}^{\l_i}\right)\left(\sum_{\s\in\fS_d}\prod_{i=1}^dy_{\s(i)}^{\l_i}\right)\\
    &=\sum_{\l\in C_d(m)}{m\choose \l}\left(\prod_{i=1}^d x_{i}^{\l_i}\right)\eta(\l) m_\l(y)
\end{align*}
where the constant $\eta(\l)$  appears since it counts how many times the monomial $\prod_{i=1}^d y_i^{\l_i}$ appears in the sum $\sum_{\s\in\fS_d}\prod_{i=1}^dy_{\s(i)}^{\l_i}.$ 

Since each composition $\l$ is a permutation of some partition of $m,$ and each monomial $\prod_{i=1}^dx_{i}^{\l_i}$ appears exactly once in the sum, the equality
\begin{align*}
    \sum_{\l\in C_d(m)}{m\choose \l}\left(\prod_{i=1}^d x_{i}^{\l_i}\right)\eta(\l) m_\l(y) &= \sum_{\l\vdash m}{m\choose \l}\eta(\l) m_\l(x)m_\l(y) \quad \text{holds.}
\end{align*}
Note that the coefficients $m\choose\l$ and $\eta(\l)$ are independent of the dimension $d.$  

Using the above computations, $K^*(x, y)$ can be written as
\begin{align*}
    K^*(x, y) &= M_m(x)^\T C M_m(y)
\end{align*}
where $M_m(x)$ is the vector of all monomial symmetric polynomials in $x$ and the matrix $C$ is diagonal with entries $\frac{1}{j!}\cdot{j\choose \l}\eta(\l)$ for each partition $\l\vdash j\leq m.$

While the monomial symmetric polynomials do form a convenient vector space basis for the ring of symmetric polynomials, they are not efficiently computable in high dimensions. In contrast to monomial symmetric polynomials, we have seen that the elementary symmetric polynomials can be efficiently computed using the recursive relationship \eqref{eq:esp-recursion}. Thus, in order to efficiently compute the kernel $K^*(x, y)$ we now need only compute the change of basis matrix from monomial symmetric polynomials to elementary symmetric polynomials indexed by partitions.

It is known that each elementary symmetric polynomial $e_\l(x)$ of degree $i$ can be written as a linear combination of monomial symmetric polynomials as follows: $e_\l(x) = \sum_{\mu\vdash i}M_{\l\mu}m_\mu(x)$ where $M_{\l\mu}$ is the number of binary matrices with row sum equal to $\l$ and column sum equal to $\mu,$ see \cite{stanley-ec2}. The quantities $M_{\l\mu}$ can  be calculated recursively due to \cite{miller-harrison}. Let $B$ be the change of basis matrix from monomial symmetric polynomials to elementary symmetric polynomials indexed by partitions and $E_m(x)$ the vector of all elementary symmetric polynomials indexed by partitions of integers up to degree $m.$ Then,
\begin{align*}
    K^*(x, y) &= E_m(x)^\T B^{-\T} C B^{-1}E_m(y).
\end{align*}
For a given kernel $K(x, y),$ the matrix $B^{-\T}CB^{-1}$ is independent of both the dimension $d$ and the individual data points and thus can be precomputed. Then, to evaluate $K^*(x, y)$ for a given pair of points $x, y$ we need only evaluate the elementary symmetric polynomials, which can be done efficiently using the above recursion, and form the elementary symmetric polynomials indexed by partitions by multiplying elementary symmetric polynomials together. When the  degree of the kernel $m$ is small, the latter is fast. 
    
\end{example}

\subsection{Set-Permutation Invariants}
We now turn our attention to set-permutation invariant kernels and build upon the theory in Section~\ref{sec:permutation-basis} to provide an efficiently computable free basis of set-permutation invariant polynomials. First, we overview the necessary background on \textit{multisymmetric polynomials}. Then, we show that \textit{$\a-$polarized elementary symmetric polynomials indexed by partitions} form a free basis of set-permutation invariant polynomials. Finally, we determine the complexity of evaluating a set-permutation invariant kernel using this basis and show that it is linear in the number of points $d.$

\begin{definition}[Multisymmetric Polynomial]
    Let $x_1, \ldots, x_d$ be $k-$dimensional vectors of indeterminates. A polynomial $f(x_1, \ldots, x_d)$ in the entries of $x_1, \ldots, x_d$ is \textit{multisymmetric} if it is invariant under the action of $\fS_d$ given by $\s\cdot f(x_1, \ldots, x_d) = f\left(x_{\s(1)}, \ldots, x_{\s(d)}\right).$
\end{definition}

Multisymmetric polynomials are the natural generalization of symmetric polynomials needed to study set-permutation invariant kernels. Like symmetric polynomials, multisymmetric polynomials are classical objects in combinatorics and invariant theory, see \cite{fleischmann-invariantTheory, smith-InvariantTheory, weyl1946classical} for an overview. We predominantly follow the discussion in \cite{fleischmann-invariantTheory}. We define \textit{multi-degree}, the natural generalization of degree to polynomials in the entries of indeterminate vectors.

\begin{definition}[Multi-Degree]
    Let $f(x_1, \ldots, x_d)=\prod_{i=1}^{d}\prod_{j=1}^k x_{ij}^{\a_{ij}}$ be a monomial in the $k-$dimensional indeterminate vectors $x_1, \ldots, x_d.$ The \textit{multi-degree} of $f$ is the $k-$dimensional vector $\a = \sum_{i=1}^d(\a_{i1}, \ldots, \a_{ik}).$ The \textit{total degree} of $f$ is $|\a| = \sum_{i=1}^d\sum_{j=1}^k\a_{ij}.$

    For a given multi-degree $\a$ and polynomial $f(x_1, \ldots, x_d),$ we denote by $f_\a$ the component of $f$ with multi-degree $\a.$
\end{definition}

We now introduce the \textit{polarization operator} $\Pol$, which maps symmetric polynomials in $d$ variables to multisymmetric polynomials in the $k-$dimensional vectors $x_1, \ldots, x_d.$

\begin{definition}[Polarization] Let $f(t_1, \ldots, t_d)$ be a polynomial in $d$ variables. The \textit{polarization} of $f$ is $\Pol(f)(x_1, \ldots, x_d) = f(x_{11}+\ldots+ x_{1k}, \ldots, x_{d1}+\ldots+x_{dk}).$
For a given multi-degree $\a,$ the $\a-$polarization of $f$ is $\Pol_\a(f)(x_1, \ldots, x_d) = f_\a(x_{11}+\ldots + x_{1k}, \ldots, x_{d1}+\ldots+x_{dk}),$ the component of $\Pol(f)$ with multi-degree $\a.$
\end{definition}

\begin{example}
    Let $d=3$ and $k=2.$ First, consider the second elementary symmetric polynomial in 3 variables, $e_2(x_1, x_2, x_3) = x_1x_2+x_1x_3+x_2x_3.$
    The polarization of $e_2$ is given by 
    \begin{align*}
        \Pol(e_2) &= (x_{11}+x_{12})(x_{21}+x_{22})+(x_{11}+x_{12})(x_{31}+x_{32})+(x_{21}+x_{22})(x_{31}+x_{32}).
    \end{align*}
    There three different multi-degrees with total degree 2 are $(2, 0), (1, 1), (0, 2)$ and each of these multi-degrees corresponds to a different multi-homogeneous component of $\Pol(e_2):$
    \begin{align*}
        \Pol_{(2,0)}(e_2)(x_1, x_2, x_3) &= x_{11}x_{21} + x_{11}x_{31} + x_{21}x_{31} \\
        \Pol_{(1, 1)}(e_2)(x_1, x_2, x_3) &= x_{11}x_{22}+x_{11}x_{32} + x_{12}x_{21}+ x_{12}x_{31}+x_{21}x_{32}+x_{22}x_{31}\\
        \Pol_{(0, 2)}(e_2)(x_1, x_2, x_3) &= x_{12}x_{22}+x_{12}x_{32}+x_{22}x_{32}.
    \end{align*}
\end{example}

It follows from the definition that polarization commutes with both addition and multiplication of polynomials:
\begin{align*}
    \Pol(f+g) &= \Pol(f)+\Pol(g), \\
    \Pol(fg) &=\Pol(f)\Pol(g).
\end{align*}
In particular, given the recursive relationship  \eqref{eq:esp-recursion}, 
\begin{align}
    \Pol(e_m)(x_1, \ldots, x_d) &= \Pol(e_m)(x_1, \ldots, x_{d-1}) + (x_{d1}+\ldots, x_{dk})\Pol(e_{m-1})(x_1, \ldots, x_{d-1}).
\end{align}
We will be particularly interested in the $\a-$polarizations of elementary symmetric polynomials. Note that a simple computation shows that the equality
\begin{align}\label{eq: polarized esp recursion}
\Pol_\a(e_m)(x_1, \ldots, x_d) &= \Pol_\a(e_m)(x_1, \ldots, x_{d-1}) + \left(\sum_{j=1}^kx_{dj} \right) \Pol_{\a-\d_j}(e_{m-1})(x_1, \ldots, x_{d-1})
\end{align}
holds, where $\d_j = (0, 0, \ldots, 0, 1, 0, \ldots, 0)$ is the indicator vector for coordinate $j,$ see \cite[Section 2.3]{fleischmann-invariantTheory}.

By \cite[Theorem 3.4.1]{smith-InvariantTheory}, the set of all $\a$-polarized elementary symmetric polynomials, with $\a$ varying over all possible multi-degrees, form an algebra basis of the space $\R[x_1, \ldots, x_d]^{\fS_d}$ of all invariant polynomials in the entries of $x_1, \ldots, x_d.$ %
It follows that taking products of $\a-$polarized elementary symmetric polynomials of degree at most $m$ generates a vector space basis of %
$\R[x_1, \ldots, x_d]_m^{\fS_d}.$ In order to avoid double-counting, we introduce \textit{vector partitions} to index the different products of $\a-$polarized elementary symmetric polynomials. 

\begin{definition}[Vector Partition]
    Fix an integer $m\geq 0.$ A $k-$vector partition of $m$ into $d$ parts is a collection of $k-$dimensional non-negative integer vectors $\Lambda=(\l_1, \ldots, \l_d)$ such that $\l_1\geq\l_2\geq\ldots\geq\l_d$ in the graded lexicographic order and $\sum_{i=1}^d\sum_{j=1}^k\l_{ij} = m$. \footnote{$\l_i\geq \l_{i'}$ in the graded lexicographic order if either $\sum_{j=1}^k\l_{ij}>\sum_{j=1}^k\l_{i'j}$ or if $\sum_{j=1}^k\l_{ij}=\sum_{j=1}^k\l_{i'j}$ and $\l_i\geq \l_{i'}$ in the lexicographic order.} Each $\l_i$ is a part of $\Lambda$ and we use the shorthand $\Lambda\vdash_{k} m$ to denote a $k$-vector partition of $m.$
\end{definition}

For a vector partition $\Lambda\vdash_k j,$ the polarized elementary symmetric polynomial indexed by $\Lambda$ is 
\begin{align*}
    E_{\Lambda}(x_1, \ldots, x_d) &= \prod_{i=1}^d\Pol_{\l_i}\left(e_{|\l_i|}\right)(x_1, \ldots, x_d)
\end{align*}
where $|\l_i| = \sum_{j=1}^k\l_{ij}.$ We now show that the set of all polarized elementary symmetric polynomials indexed by vector partitions of total degree at most $m$ form a free basis of invariants for the sequence  $\cF_m = \{(\R[V_d]_m, \gamma_d)\}_{d\in\N}.$

\begin{proposition}[Free Basis of Set-Permutation Invariant Polynomials]
    Let $V_d$ and $G_d$ be the vector spaces and groups from Example \ref{thm:Set-permutations}. Consider the consistent sequence $\cF_m=\{(\R[V_d]_m, \gamma_d)\}_{d\in\N}$ of $\{G_d\}_{d\in\N}-$representations where $\gamma_d$ is the zero-padding map. The set $\{E_\Lambda(x_1, \ldots, x_d) : \Lambda\vdash_k j\leq m\}$ of all polarized elementary symmetric polynomials indexed by $k$-vector partitions of all $j \leq m$ forms a free basis of invariants for $\cF_m.$
\end{proposition}

\begin{proof}
    We first show that the set of all polarized elementary symmetric polynomials indexed by vector partitions forms a basis by showing that the number of $k-$vector partitions of integers at most $m$ is exactly equal to the number of distinct products of $\a-$polarized elementary symmetric polynomials with degree at most $m$ and is equal to $\dim\R[V_d]_{m}^{G_d}$ as given in Example~\ref{thm:Set-permutations}.

    First, recall that our proof of Example \ref{thm:Set-permutations} amounted to choosing unique $G_d-$orbit representatives of non-negative integer $d\times k$ matrices whose entries totaled at most $m.$ We previously argued that each orbit contains at least one matrix whose row sums form a partition of some integer $j\leq m.$ Note that there is exactly one matrix in each orbit such that the rows are ordered in the graded lexicographic order. This matrix is in direct correspondence with a $k-$vector partition and thus it follows that the number of $k-$vector partitions of integers at most $m$ is equal to $\dim\R[V_d]_m^{G_d}.$

    Since there is one elementary symmetric polynomial for each degree $j\leq m,$ the number of $\alpha-$polarized elementary symmetric polynomials of degree $j$ is equal to the number of distinct multi-degree vectors $\alpha$ with total degree $|\alpha|=j.$ Note that each multi-degree vector is in direct correspondence with a weak composition of $j$ and so there are exactly ${j+k-1\choose k-1}$ distinct $\alpha-$polarized elementary symmetric polynomials of degree $j.$ To generate a vector space basis of $\R[x_1, \ldots, x_d]_m^{\fS_d},$ we take products of $\alpha-$polarized elementary symmetric polynomials such that the total degree is at most $m.$ We count the number of distinct products as follows: for each partition $\alpha=\{1^{\mu_1}, \ldots, j^{\mu_j}\}$ of $j\leq m,$ choose $\mu_i(\alpha)$ polarized elementary symmetric polynomial of degree $i$ for all $i=1, \ldots, j.$ The product of the chosen polarized elementary symmetric polynomials is an $\fS_d-$invariant polynomial of degree $j.$ For each $i,$ there are ${c_k(i)+\mu_i(\alpha) - 1\choose \mu_i(\alpha)}$ possible choices for the polarized elementary symmetric polynomials of degree $i.$ Taking the sum over all possible degrees, there are $\sum_{j=0}^m\sum_{\alpha\vdash j}\prod_{i=1}^j{c_k(i) + \mu_i(\alpha)-1\choose \mu_i(\alpha)}$ distinct products of polarized elementary symmetric polynomials of degree at most $m.$ This is exactly the dimension of $\R[x_1, \ldots, x_d]_m^{\fS_d}$ as given in Theorem \ref{thm:Set-permutations} and so the set of all products of $\alpha$-polarized elementary symmetric polynomials with total degree at most $m$ forms a basis for $\R[x_1, \ldots, x_d]_m^{\fS_d}$. 

    Since $\gamma_d$ is the zero-padding map, for each $f(x_1, \ldots, x_d)\in\R[V_d]_{m}$ we have that the projection onto $\R[V_{d-1}]_m$ is $\gamma_d^*f(x_1, \ldots, x_d) = f(x_1, \ldots, x_{d-1}, 0).$ So, for each polarized elementary symmetric polynomial $\Pol_\alpha(e_j)(x_1, \ldots, x_d)$ we have that
    \begin{align*}
        \gamma_d^*\Pol_\alpha(e_j)(x_1, \ldots, x_d) &= \Pol_\alpha(e_j)(x_1, \ldots, x_{d-1}, 0) \\
        &=\Pol_\alpha(e_j)(x_1, \ldots, x_{d-1}) + \sum_{i=1}^k0\cdot \Pol_{\alpha-\d_i}(e_{j-1})(x_1, \ldots, x_{d-1}) \\
        &=\Pol_\alpha(e_j)(x_1, \ldots, x_{d-1})
    \end{align*}
    and so the polarized elementary symmetric polynomials are free elements. Taking the relevant products, it follows that polarized elementary symmetric polynomial indexed by vector partitions form a free basis of invariants.
\end{proof}

For any given dimensions $d, k$, and $r = \dim(\R[V_d]_m^{\fS_d})$, let $P_m^{(d)} \colon \R^{d\times k} \rightarrow \R^r$ %
be a mapping producing a vector of all evaluations of $\a-$polarized elementary symmetric polynomials indexed by vector partitions of degree at most $m.$ As before we will drop the index $d$ and simply write $P_m.$  Suppose that one has access to the matrix $C$ representing a kernel $K$ with respect to the basis of polarized elementary symmetric polynomials indexed by vector partitions. Then, for data points $x, y\in\R^d,$ Procedure \ref{alg:set-permutation} describes the evaluation of $K(x, y).$

\begin{algorithm}
    \caption{Set-Permutation Invariant Kernel Evaluation}\label{alg:set-permutation}
    \begin{algorithmic}[1]
    \REQUIRE Data points $x, y\in \R^{d\times k},$ matrix $C$ representing $K.$
    \STATE Set $P_m(x) = P_m(y) = 0 \in \R^r$
    \FOR{$j=0$ to $m$}\label{step:pol-esp}
    \STATE Compute $\Pol(e_j)(x)$
    \STATE Compute $\Pol(e_j)(y)$
    \FORALL{$\a\in C_k(j)$}
    \STATE Store $\Pol_\a(e_j)(x)$
    \STATE Store $\Pol_\a(e_j)(y)$
    \ENDFOR
    \ENDFOR
    \FOR{$j=0$ to $m$}\label{step:pol-esp-parts}
    \FORALL{$\Lambda\vdash_k j$}
    \STATE Set $P_m(x)_\Lambda = \prod_{i=1}^d \Pol_{\l_i}(e_{|\l_i|})(x)$
    \STATE Set $P_m(y)_\Lambda = \prod_{i=1}^d \Pol_{\l_i}(e_{|\l_i|})(y)$
    \ENDFOR
    \ENDFOR
    \ENSURE $P_m(x)^\T C P_m(y).$
    \end{algorithmic}
\end{algorithm}

 We use Procedure \ref{alg:set-permutation} to evaluate set-permutation kernels in the numerical experiments; see Section \ref{app:classification-experiment}.

\begin{proposition}[Complexity of evaluating a set-permutation invariant kernel]\label{prop:set-complexity}
Suppose that $\{K_d:\R^{d\times k}\times\R^{d\times k} \rightarrow \R\}$ is a free set-permutation invariant kernel. Suppose one has access to the matrix $C$ representing $K_d$ with respect to the basis of polarized elementary symmetric polynomials. Then, the number of floating point operations required to compute $K_d(x, y)$ via Procedure \ref{alg:set-permutation} is upper bounded by
    $$(m+1)\cdot\exp(\sqrt{2m/3})\cdot(k-1)^{mk}+\exp(2\sqrt{2m/3})\cdot(k-1)^{2mk} + \frac{d(2k-1)(k-1)^m}{m!}.$$
\end{proposition}

\begin{proof}
We upper bound the number of flops necessary to compute $K_d(x, y)$ as follows.

\begin{enumerate}
    \item Completing the loop in Step \ref{step:pol-esp} requires evaluating all $\a-$polarized elementary symmetric polynomials.  Given the recursive relationship \eqref{eq: polarized esp recursion}, evaluating all $\a-$polarized elementary symmetric polynomials entails filling the entries of a $d\times {m+k-1\choose k-1}$ matrix since there are ${m+k-1\choose k-1}$ multi-degrees with total degree $m.$ Determining each entry requires $k$ additions and $k-1$ products and so in total evaluating all $\a-$polarized elementary symmetric polynomials takes $d\cdot{m+k-1\choose k-1}\cdot(2k-1)$ flops. Note that ${m+k-1\choose k-1}={m+k-1\choose m}\leq\frac{(k-1)^m}{m!}.$
    \item Completing the loop in Step \ref{step:pol-esp-parts} requires computing all distinct products of $\a-$polarized elementary symmetric polynomials of degree at most $m.$ Recall that the quantity $\dim\R[V_d]_m^{G_d}$ grows at most $\exp(\sqrt{2m/3})\cdot(k-1)^{mk}$ and so there are at most $\exp(\sqrt{2m/3})\cdot(k-1)^{mk}$ polarized elementary symmetric polynomials indexed by partitions to evaluate. Evaluating each polarized elementary symmetric polynomial indexed by vector partition requires taking the product of $m$ $\a-$polarized elementary symmetric polynomials. It follows that the number of flops required to complete the loop in Step \ref{step:pol-esp-parts} is $m\exp(\sqrt{2m/3})\cdot(k-1)^{mk}.$
    \item Evaluating the output $P_m(x)^\T CP_m(y)$ via naive matrix multiplication takes $(r+1)(2r-1)$ operations and so the number of operations required is at most $(\exp(\sqrt{2m/3})\cdot(k-1)^{mk}+1)(2(\exp(\sqrt{2m/3})\cdot(k-1)^{mk})-1)$.
\end{enumerate}
Taking the sum of the number of operations at each step yields the stated upper bound. Note that while there is an unavoidable exponential dependence on both $k$ and $m,$ but a linear dependence on the dimension $d.$
\end{proof}

The proof of Proposition \ref{prop:set-complexity} also completes the proof of Theorem \ref{thm:complexity}. We now introduce a specific example of a set-permutation invariant kernel that can be evaluated efficiently in high dimensions.

\begin{example}[Set-Permutation Invariant Taylor Features]\label{ex:set-tf} We now build on Illustration \ref{ex: invariant tf} and provide efficient methods to compute kernels that are invariant under set permutations. 

While there is an unavoidable exponential dependence on $k$ in the complexity of evaluating the basis of polarized elementary symmetric polynomial indexed by vector partitions, there are set-permutation invariant kernels for which computation is much quicker. For instance, consider the kernel defined by
\begin{align*}
    K(x, y) &= \frac{1}{|\fS_d|}\sum_{\s\in\fS_d}\sum_{j=0}^m\frac{\langle x_1, \s(y_1)\rangle^j + \ldots +\langle x_k, \s(y_k)\rangle^j}{j!}
\end{align*}
where $x_i$ is the $i^{th}$ \textit{column} of the matrix  $X\in\R^{d\times k}.$ Note that $K$ is the sum of $k$ invariant Taylor features kernels as described in Example~\ref{ex: invariant tf}. Since the complexity of evaluating a permutation invariant kernel is linear in dimension $d,$ and $K$ is evaluated by evaluating $k$ permutation invariant kernels, the complexity of evaluating $K$ is linear in both $d$ and $k.$ 

However, it is clear that the above kernel is not full rank; the rank of $K$ is $k\cdot\sum_{j=1}^mp(j)$ which is in general less than the maximum rank of a set-permutation invariant kernel given in Example \ref{thm:Set-permutations}. With this in mind, consider instead the invariant Taylor features kernel for sets given by 
\begin{align*}
    K(x, y) &= \frac{1}{|\fS_d|}\sum_{j=0}^m\sum_{\s\in\fS_d}(\langle x_1, \s(y_1)\rangle+\ldots+\langle x_k, \s(y_k)\rangle^j\\
    &=\frac{1}{|\fS_d|}\sum_{j=0}^m\sum_{\l\in C_k(j)}\sum_{\s\in\fS_d}\prod_{i=1}^k\langle x_i, \s(y_i)\rangle^{\l_i}.
\end{align*}
We would like to be able to write the above kernel in terms of the invariant inner products $\sum_{\s\in\fS_d}\langle x_i, \s(y_i)\rangle^{\l_i},$ since we have seen in Illustration \ref{ex: invariant tf} that these inner products are easy to compute. Although the sum over $\s\in\fS_d$ does not commute with the product, we can approximate the invariant Taylor features kernel for sets with the kernel defined by
\begin{align}\label{eq:set-kernel}
    K(x, y) &= \frac{1}{|\fS_d|}\sum_{j=0}^m\sum_{\l\in C_k(j)}\prod_{i=1}^k\sum_{\s\in\fS_d}\langle x_i, \s(y_i)\rangle^{\l_i}.
\end{align} 
In Section \ref{app:classification-experiment}, we use the kernel in \eqref{eq:set-kernel} to classify a mixture of Gaussians and compare its performance to the non-invariant Taylor features kernel. 
\end{example}

\section{Numerical Experiments}\label{sec:numerics}

In this section, we perform three numerical experiments that support the theory in this paper. Our first experiment showcases the use of the invariant kernels defined in \eqref{eq:set-kernel} to classify sets of points. For our second experiment, we use the invariant regression estimator \eqref{eq:GOAT} to learn sequences of symmetric polynomials that generalize well across dimensions. Our final experiment uses the formula in Theorem \ref{thm:calculate_dim_intro} to estimate the dimensions of different spaces of invariant polynomials where the true value is known. We wrote these experiments in Python with some key functions imported from SageMath. The code and data to reproduce these results can be found at
\begin{center}
    \url{https://github.com/j4ck-k/invariant-kernels}.
\end{center}
\subsection{Classifying Sets of Points}\label{app:classification-experiment}
The kernel $K$ defined in \eqref{eq:set-kernel} is an example of an invariant kernel under set permutations.  We use this kernel to classify sets of points; the data points in each set are drawn from a fixed Gaussian distribution. We sample from a collection of seven mean-zero Gaussian distributions. Each covariance matrix is diagonal, with diagonal entries in $\{1, 2, 0.5\}.$ %

We run multiple experiments, for each of them approximately half of the sets have been sampled from the standard normal distribution, and the task is to correctly classify which data points are in this class. Since the distribution does not depend on the order in which the points were sampled, the classifier should be invariant under reshuffling of the data; therefore, an invariant kernel $K$ is appropriate for this task. 

Let $\{(x_i, y_i)\}_{i=1}^n$ be the training data set---$x_i$ is the stacked set of Gaussian samples and $y_i\in\{-1, 1\}$ encodes whether or not the set was drawn from a standard Gaussian---and $M_K = [K(x_i, x_j)]_{i,j=1}^n$ the Gram matrix corresponding to this data. We construct a classifier by first computing $\a^*$ by solving
\begin{align*}
    \a^{\star} \in \mathop{\argmin}_{\a\in\R^n}\sum_{i=1}^n \frac{1}{2}\|y_i - (M_K\a)_i\|^2 +\frac{\l}{2}\a^\T M_K \a, \quad \text{or, equivalently,} \quad \alpha^{\star} =(M_K + \l I_n)^{-1}y
\end{align*}
where the regularizing parameter $\l$ is chosen via $4-$fold cross validation.\footnote{Note that we  are using the $\ell_2$ loss function for ease of computation.} The classification function is then given by
\begin{align*}
    f(x) &= \sgn\left(\sum_{i=1}^n \a_i K(x, x_i)\right).
\end{align*}

We run experiment using training sets of size $n\in \{200, 400, 600, 800, 1000\}$ %
Each classifier is then tested using an unseen set of $200$ data points, and its success is measured based on the accuracy of the classifications on the test set. The accuracy of the classifier on the test set $\{(x_i, y_i)\}_{i=1}^n$ is calculated using the formula
\begin{align*}
    {\rm Accuracy}(f) &= \frac{\Card\{i: f(x_i) = y_i\}}{n}.
\end{align*}

To construct the training and test sets, we sample the sets of points from a mixture of Gaussians with probability 0.5 of being sampled from the standard normal distribution and equal probability of being sampled from the remaining distributions.

 We compare the performance of the invariant kernel $K$ and non-invariant Taylor features approximation to the Gaussian kernel. For this comparison, we run two batches of experiments summarized in Figures~\ref{fig:classification-var-degree} and~\ref{fig:classification-var-k}.

\textbf{Setting 1: Classifying points in $\R^2$ with Kernels of Varying Degree.} We first fix the dimension $k=2$ and test invariant and non-invariant kernels of degree $m \in \{1, 2, 3, 4, 5\}.$ As shown in Figure \ref{fig:classification-var-degree}, invariant kernels of degrees higher than one far outperformed their non-invariant counterparts. Test accuracies of nonlinear invariant kernels were approximately 98\%, whereas all non-invariant kernels had accuracy below 50\%, which is worse than if the function simply randomly assigned classifications to each point cloud. Moreover, the linear invariant kernel also had much greater success than all of the non-invariant kernels. This demonstrates that for this task, even the simplest invariant kernel is more powerful than any non-invariant kernel for this task. We highlight that increasing the degree of both the invariant and non-invariant kernels had little effect on the accuracy of the predictions, and thus, for this task, it suffices to consider only the relatively simple quadratic kernels.

\textbf{Setting 2: Classifying point in $\R^k$ with Quadratic Kernels.} We now fix the degree $m=2$ of both the invariant and non-invariant kernel and test their performance on sets of points in dimension $k \in \{2, 3, 4, 5\}.$ As shown in Figure \ref{fig:classification-var-k}, the invariant kernel once again had almost perfect classification accuracy, whereas the non-invariant kernel had an accuracy of approximately $50\%,$ which is equivalent to the accuracy of a random guess.

\begin{figure}[h]
    \centering
    \begin{subfigure}[T]{0.49\textwidth}
         \centering
         \includegraphics[width=\textwidth]{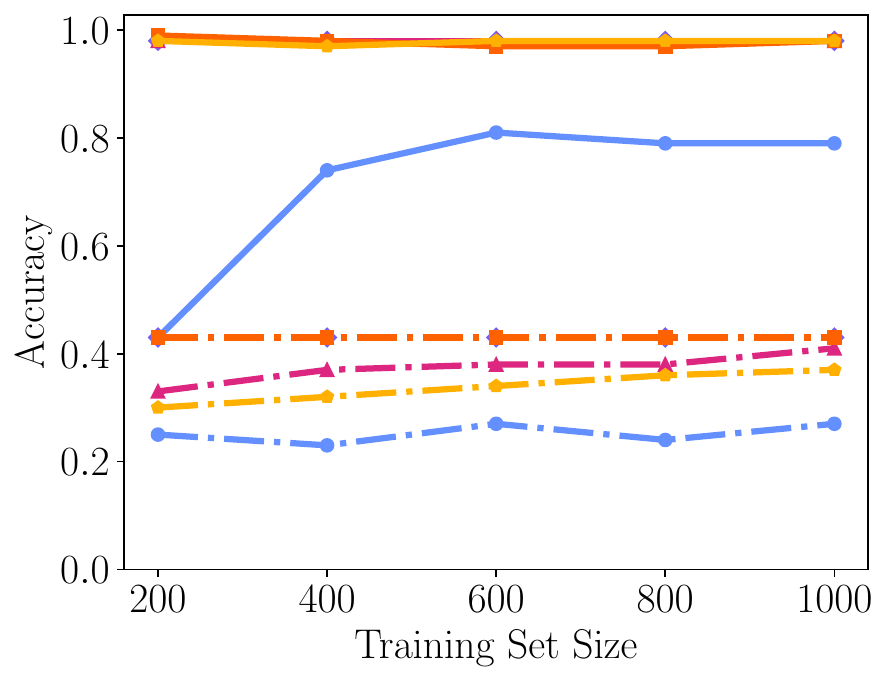}
         \includegraphics[width=\textwidth]{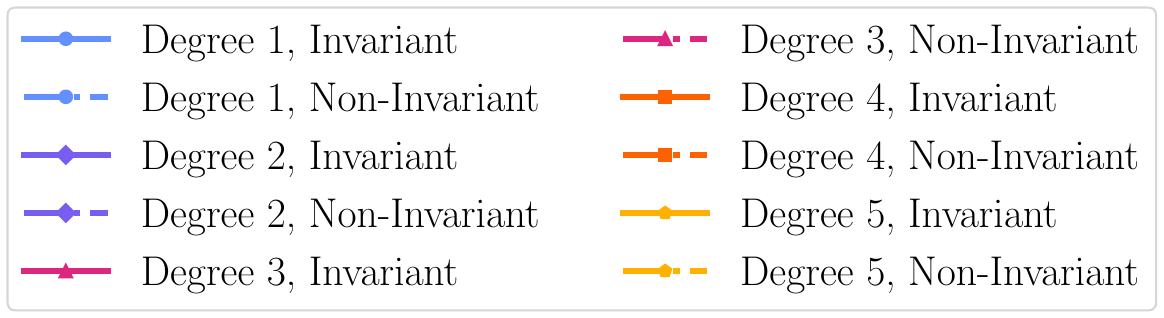}
         \caption{Accuracy with Varying Degree $m$}
         \label{fig:classification-var-degree}
     \end{subfigure}
     \hfill
     \begin{subfigure}[T]{0.50\textwidth}
         \centering
         \includegraphics[width=\textwidth]{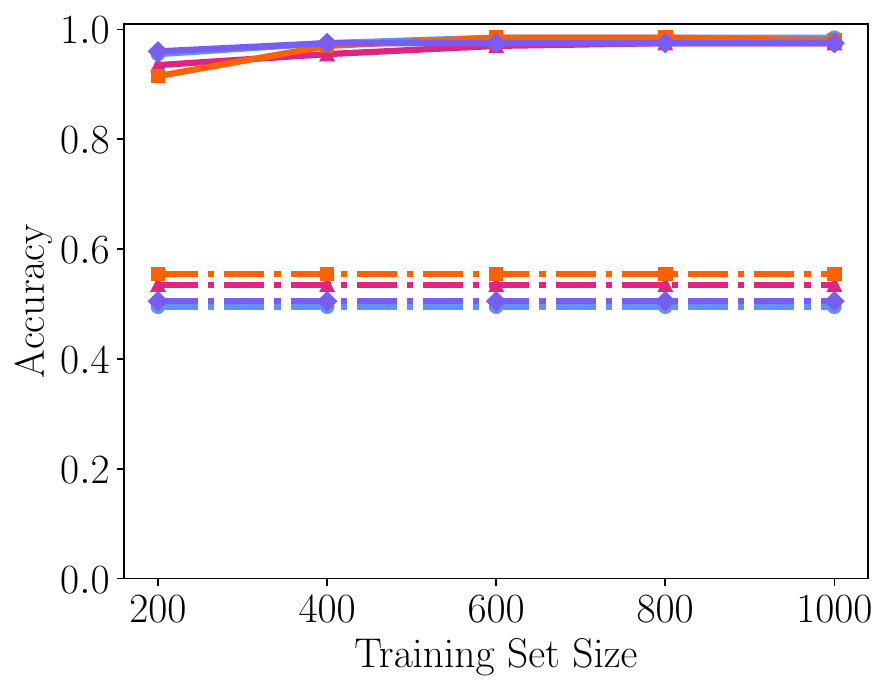}
         \includegraphics[width=\textwidth]{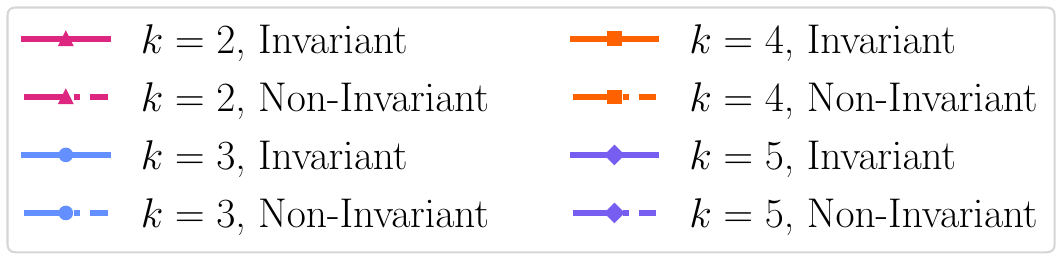}
         \caption{Accuracy with Varying Dimension $k$}
         \label{fig:classification-var-k}
     \end{subfigure}
     \caption{
     Comparison of performance between $(i)$ the set-permutation invariant kernel \eqref{eq:set-kernel} and $(ii)$ the non-invariant Taylor features kernel for classifying points drawn from a mixture of Gaussians (Illustration \ref{ex: invariant tf}). We present two experimental scenarios: Figure \ref{fig:classification-var-degree} (left) shows results with fixed dimension $k=2$ across varying degrees $m \in \{1, \dots, 5\}$, while Figure \ref{fig:classification-var-k} (right) shows results with fixed degree $m=2$ across varying dimensions $k \in \{1, \dots, m\}$. In each setting, the invariant kernel outperforms the non-invariant kernel.}
\end{figure}

\subsection{Generalization Across Dimensions} \label{app:regression-experiment} We now leverage the theory in Section \ref{sec:generalization} to learn sequences of symmetric polynomials on data across different dimensions. We consider symmetric polynomials of degrees $m \in \{2, 3, 4, 5\}$. For each degree $m$, we learn ten randomly generated symmetric polynomials, and so in total we learn $40$ different symmetric polynomials. Each polynomial is a random linear combination of elementary symmetric polynomials indexed by partitions with coefficients sampled from the standard Gaussian distribution. For each polynomial $f,$ our training set consists of $10,000$ points $\{x_i\}\subset\R^{100}$ along with the noisy outputs $\{f(x_i) + \e_i\}$ where $x_i$ is sampled from the standard Gaussian distribution and $\e_i$ is randomly sampled from a Gaussian distribution with mean zero and variance $\sigma^2 = 10^{-2}.$ We construct a predictor as in \eqref{eq:GOAT} and then test its performance as the dimension of the test set increased from $d=100$ to $d=1000.$ 

Figure \ref{fig:generalization} shows the mean squared error (MSE) and mean percentage error (MPE) of each learned predictor function as the dimension of the test sets increased. The MSE and MPE on the set $\{(x_i, y_i)\}_{i=1}^n$ for each predictor were calculated according to the formulas
\begin{equation*}
    {\rm MSE}(f) = \frac{1}{n}\sum_{i=1}^n(y_i-f(x_i))^2 \quad \text{and} \quad
    {\rm MPE}(f) = \frac{1}{n}\sum_{i=1}^n\frac{|y_i-f(x_i)|}{y_i}.
\end{equation*}

\begin{figure}[h]
    \centering    \includegraphics[width=0.49\linewidth]{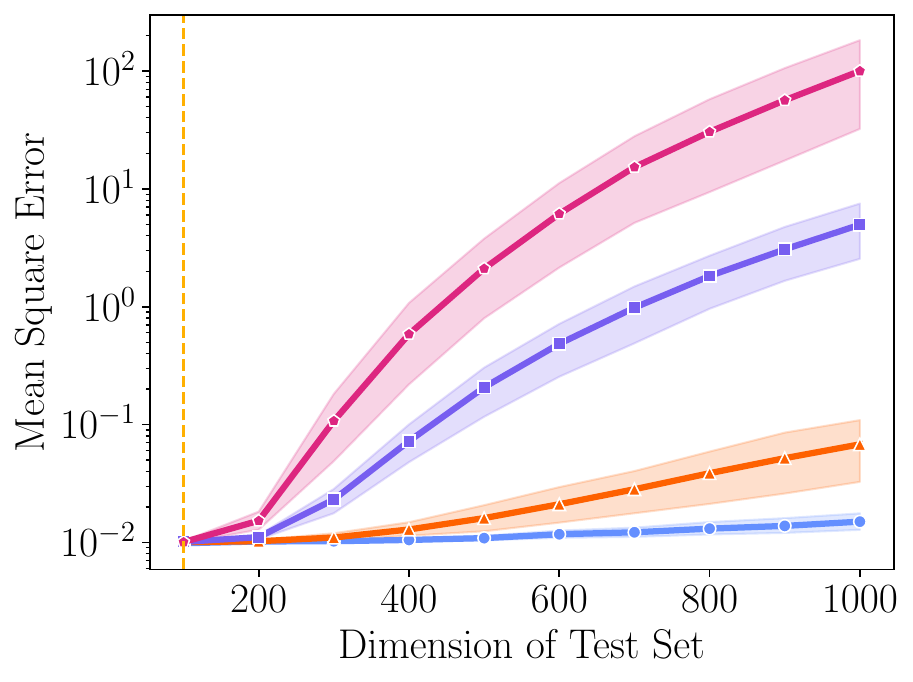}\hfill
\includegraphics[width=0.49\linewidth]{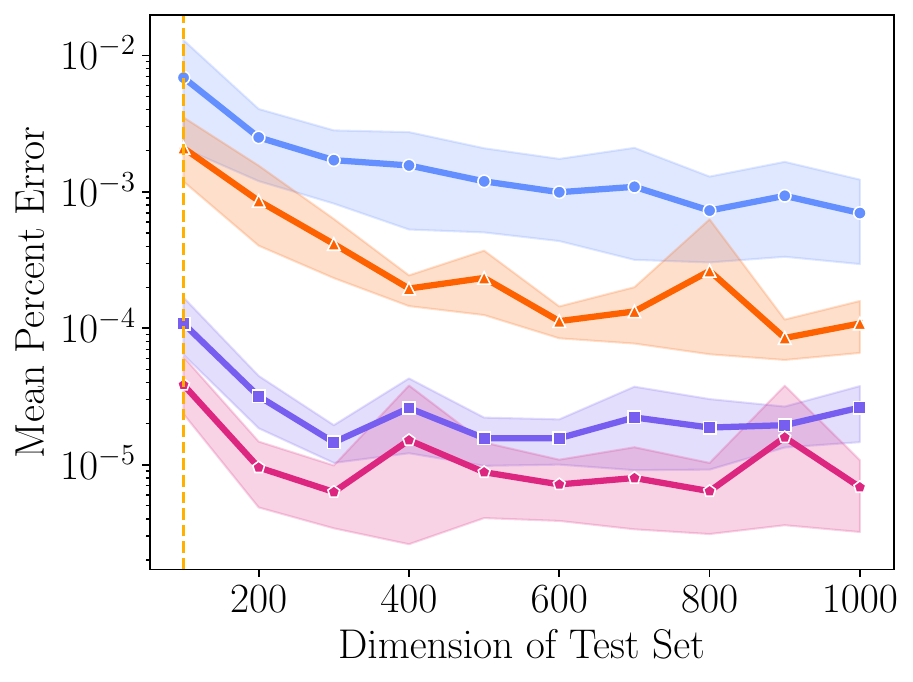}
\includegraphics[width=0.6\textwidth]{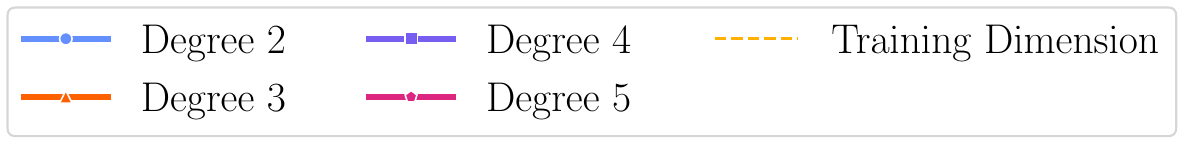}
    \caption{Mean Square Error (left) and Mean Percentage Error (right) versus test dimension using the estimator defined \eqref{eq:informal-estimator}. For each degree, data was collected from ten randomly generated symmetric polynomials. Each estimator was trained on 10,000 noisy data points in dimension $d=100$ and tested on 200 noisy data points in dimensions $d=100, 200, \ldots, 1000.$ The highlighted regions indicate a 95\% confidence interval for the values.}
    \label{fig:generalization}
\end{figure}

The mean MSE stays close to zero for each predictor of degree $m\in \{2, 3, 4\},$ but rapidly increased for degree $m=5$ and dimensions greater than $d=500.$ However, the MPE of the predictors consistently decreased as both degree $m$ and test set dimension $d$ increased, with all predictors having less than a $1\%$ error in dimension $d=1000.$

\subsection{Estimating $\dim(\R[V]_m^G)$ Using Theorem \ref{thm:calculate_dim_intro}}\label{app:monte-carlo-experiment} Figure \ref{fig:dim-estimates} displays Monte Carlo estimates of the dimension of $\R[V]_m^G$ for three different choices of vector space $V$ and group $G.$ The estimates were formed by calculating an empirical average of the formula in Theorem \ref{thm:calculate_dim_intro} with an increasing number of samples. Each experiment was repeated 10 times and Figure~\ref{fig:dim-estimates} plots the average estimate and the 95\% confidence interval.

\begin{figure}[h]
    \centering
    \includegraphics[width=0.325\linewidth]{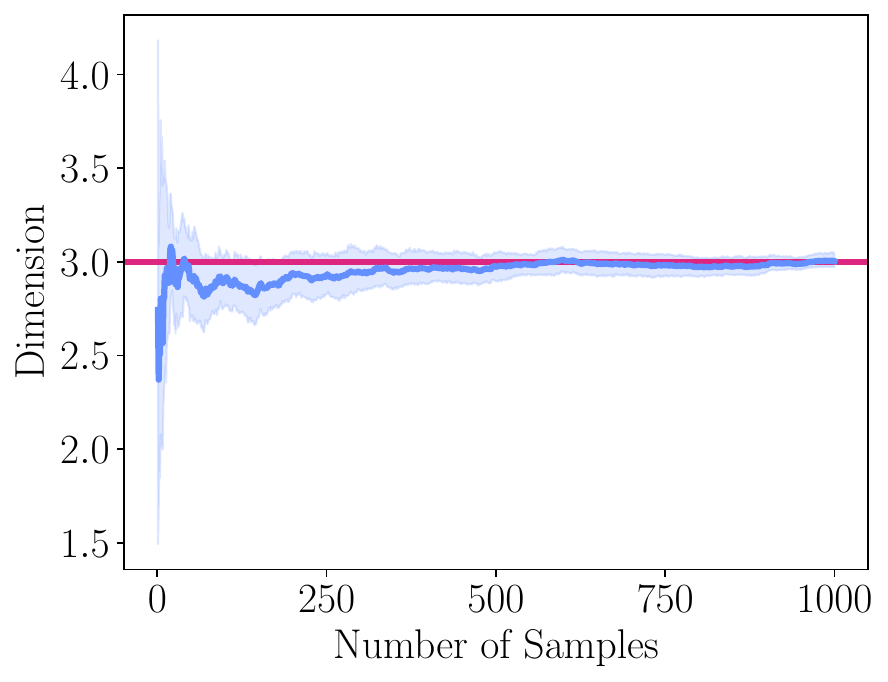}
    \includegraphics[width=0.315\linewidth]{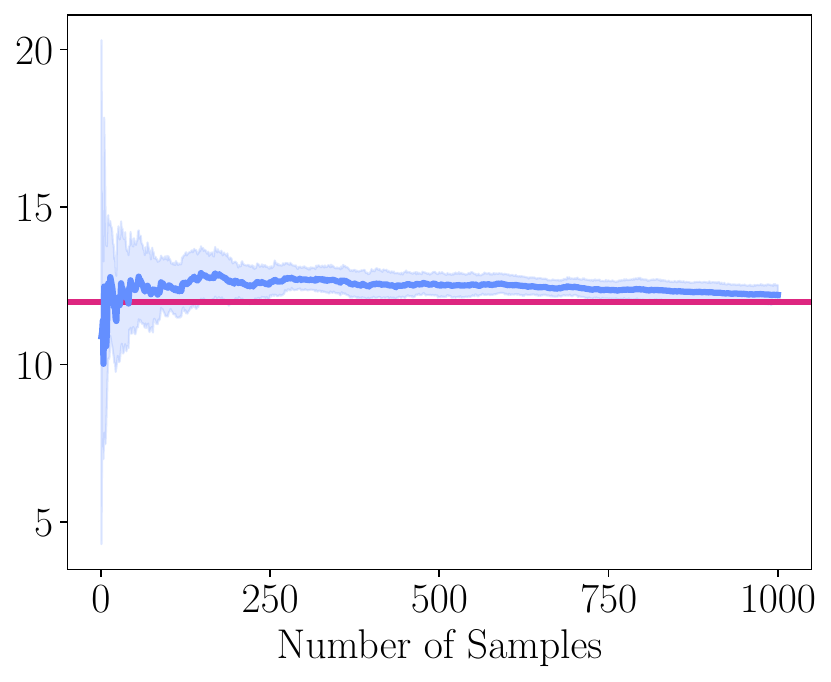}
    \includegraphics[width=0.33\linewidth]{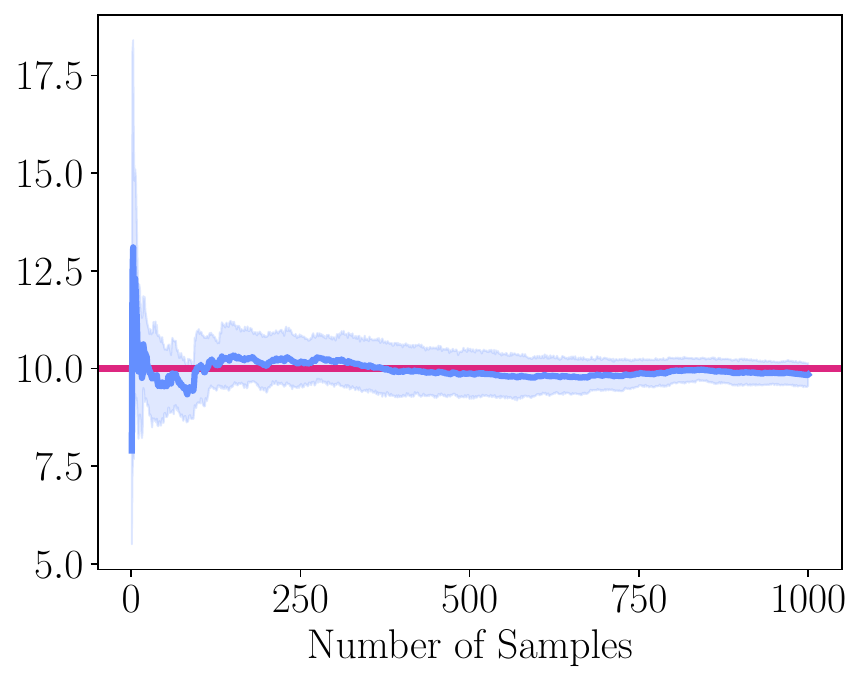}
    \includegraphics[width=0.315\textwidth]{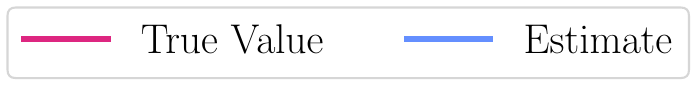}
    \caption{Monte Carlo estimates of $\dim(\R[V]_4^G)$ for orthogonal transformations, coordinate permutations and graph permutations. The estimates are calculated using the empirical average approximation to the formula in Theorem \ref{thm:calculate_dim_intro} with an increasing number of samples. In each case, we see that the estimate stabilises around the correct dimension relatively quickly. The highlighted regions are the 95\% confidence intervals for the estimate. %
    }
    \label{fig:dim-estimates}
\end{figure}

The first plot shows the estimate for the dimension of the space $\R[V]_4^G$ where $V=\R^d$ and $G = O(d),$ the group of orthogonal transformations. To construct the estimates of the dimension, we randomly sampled orthogonal matrices using the algorithm described in \cite{mezzadri-random-mats}. As any polynomial that is invariant under orthogonal transformations must be a polynomial in the square norm, the dimension $d$ is arbitrary. Moreover, we can directly calculate that the true dimension is $\dim(\R[V]_4^G) = 3.$ As shown in Figure \ref{fig:dim-estimates}, the estimated dimension quickly stabilizes to be approximately 3 as the number of samples increases.

The second plot shows the estimate for the dimension of $\R[V]_4^G$ where $V=\R^d$ for $d\geq 4$ and $G=\fS_d$ acts via coordinate permutation. Random elements of the group $\fS_d$ were sampled using the \texttt{random\_element} function in Sage. Our estimate quickly approaches the true value computed in Example \ref{thm:Permutations}, i.e., $\dim(\R[V]_4^G)=\sum_{j=0}^4p(j) = 12.$ 

The final plot shows the estimate for the dimension of $\R[V]_2^G$ where $V\subset \R^{d\times d}$ is the space of all adjacency matrices of graphs on $d$ vertices and $G=\fS_d$ acts via simultaneous permutation of rows and columns. As shown in Example \ref{theorem: graphs rank}, the rank of kernels on graphs does not stabilize until the number of nodes $d$ is at least twice the degree $m$ of the kernel. Thus, for ease of computation, we use a smaller degree for this example. As in the previous example, random elements of the group $\fS_d$ were sampled using Sage's \texttt{random\_element} function. The exact value of this dimension $\dim(\R[V]_2^G)=10$ was calculated by summing the first three values in the sequence \cite{oeisA007717}. 

\vspace{-0.2in} 

\section{Conclusion}
In this work, we studied $G-$invariant polynomial kernels of fixed degree and proved that the rank of these kernels is independent of the input data dimension for important examples such as unordered sets, graphs, and point clouds. We provided elementary combinatorial arguments for these facts and also gave a general formula for computing invariant kernel ranks by leveraging representation theory. Moreover, we highlighted a direct connection between rank stabilization of invariant kernels and the phenomenon of representation stability. The theory of representation stability in turn allows us to learn sequences of polynomials that generalize well to data of arbitrary dimension. In particular, we provided a minimax optimal regression procedure for cross-dimensional generalization. Finally, we provided efficiently computable bases for the spaces of permutation invariant and set-permutation invariant polynomials, thereby enabling computation in high dimensions.
Experimental results validated the developed theory.

\newpage
\printbibliography

\appendix

\section{Proofs from Section~\ref{section: preliminaries}}\label{app:prelim-proofs}
\subsection{Proof of Lemma~\ref{lem:inv_kern}}\label{sec:proof_lem_inv}
The fact that $K^*$ is invariant follows directly from right-translation invariance of the Haar measure, $\mu(Sg)=\mu(S)$, for all $g\in G$ and all Borel sets $S$. Suppose now that $K$ is PSD and fix 
arbitrary points $x_1, \ldots, x_n\in V$.  Using linearity of the integral, we compute 
\begin{align*}
    \sum_{i,j=1}^n c_ic_j K^*(x_i, x_j)
    = \int_G\int_{G} \underbrace{\sum_{i,j=1}^n c_ic_j K(gx_i, hx_j)}_{\geq 0}\, d\mu(g)d\mu(h),
\end{align*}
and therefore $K^*$ is PSD.

\subsection{Proof of Lemma~\ref{lemma: C is unique}}\label{sec:proof_lembilin}
Writing the polynomial kernel $K$ in terms of the basis $f_1, \ldots, f_N$  yields the equation 
$$ K(x,y)=\sum_{i=1}^N\sum_{j=1}^N c_{ij} f_j(x)f_i(y),$$
for some constants $c_{ij}$. Let us now argue uniqueness. To this end, suppose that there exists another matrix $\tilde C$ satisfying \eqref{eqn:K_as_lin_form}. Then we may write $K(x,y)$ in two ways:
$$\sum_{i=1}^N\left(\sum_{j=1}^N c_{ij} f_j(x)\right) f_i(y)=K(x,y)=\sum_{i=1}^N\left(\sum_{j=1}^N \tilde c_{ij} f_j(x)\right)f_i(y).$$
Linear independence of $f_i$ implies the equality $\sum_{j=1}^N c_{ij} f_j(x)=\sum_{j=1}^N \tilde c_{ij} f_j(x)$ for any $x$ and any index $i=1,\dots,N$. Again using linear independence, we deduce $c_{ij}=\tilde c_{ij}$ for all $i,j$, as claimed. Finally, symmetry of $K(x,y)$ in $x$ and $y$ implies that \eqref{eqn:K_as_lin_form} remains true with $C$ replaced by $(C+C^\top)/2$. Uniqueness of the matrix $C$ satisfying \eqref{eqn:K_as_lin_form}, which we already established, therefore implies that $C$ is symmetric. Taking the expectation $\E_{g,h}$ of both sides in \eqref{eqn:K_as_lin_form} directly yields \eqref{eqn:K_as_lin_form_sym}.

\subsection{Proof of Theorem~\ref{theorem: rank is number of orbits}} 
 Define $\mathcal{F}(x)=[f_1(x),\ldots,f_N(x)]$, $\mathcal{F}_*(x)=[f_1^*(x),\ldots,f^*_N(x)]$, and
$\mathcal{V}(x)=[g_1(x),\ldots, g_r(x)]$ 
where $\{g_j\}_{j=1}^r$ is a basis for ${\rm span}\{f_1^*, \ldots, f_N^*\}$. Then expanding $f_i^*$ in the $\{g_j\}$ basis we deduce that there exists a matrix $B\in \R^{r\times N}$ such that  $\cF_*(x)=\mathcal{V}(x)B$ for all $x\in\R^d$. Consequently using \eqref{eqn:K_as_lin_form}, we deduce 
\begin{align*}
K(x,y)&=\mathcal{V}(x)\cdot (BCB^\top) \cdot\mathcal{V}(x)^\top.
\end{align*}
Forming a rank $r$ factorization $BCB^\top=LR^\top$ for some $L,R\in\R^{r\times r}$ therefore yields the equality $K(x,y)=\langle \mathcal{V}(x)L,\mathcal{V}(x)R\rangle$, thereby verifying the inequality $\rank K\leq r$ in \eqref{eqn:rank_est_basis}. 

Next, we show the equality in \eqref{eqn:rank_est_basis}. Clearly, since $f_i^*$ lie in $\R[V]_m^G$, the inequality holds:
\begin{equation}\label{eqn:invar_space_lower}
r\leq\dim(\R[V]_m^G).
\end{equation}

It remains to argue that $\dim(\R[V]_m^G)\leq r.$ Choose $h(x) \in\R[V]_m^G.$ Since $\R[V]_m^G$ is a subspace of $\R[V]_m,$ it follows that there exist coefficients $c_1, \ldots, c_N$ such that $h(x) = \sum_{i=1}^Nc_i f_i(x).$ Then, as $h$ is invariant under the action of $G$ and $G$ acts linearly on $\R[V]_m$, the equalities 
\begin{align*}
    h(x) &= h^*(x) = \sum_{i=1}^Nc_if_i^*(x)
\end{align*}
hold. It follows that $f_1^*(x), \ldots, f_N^*(x)$ form a spanning set for $\R[V]_m^G$ and thus the inequality $\dim(\R[V]_m^G)\leq r$ holds as claimed.

\section{Proofs from Section~\ref{sec:generalization}} \label{app:generalization}
\subsection{Proof of Theorem~\ref{thm:bound_estimator}}
The proof follows along similar lines as the standard argument for linear regression in a fixed dimension \cite[Proposition 3.5]{bach2024learning}.
The fact that the estimator is unbiased follows immediately from its definition and the fact that $y = \Phi(X)\alpha^{\star} + \varepsilon$ with $\EE\varepsilon = 0.$  
Take any $\alpha \in \R^d$, expanding and using the fact that the noise is i.i.d, we derive 
\begin{align}\begin{split}
R_{T}^{(\alpha^{\star})}(\alpha) &= \frac{1}{|T|}  \sum_{(x, y) \in T} \EE\left( \dotp{\varphi(x),  \alpha} - \dotp{\varphi(x), \alpha^{\star}} + \varepsilon \right)^2 \\
&= \sigma^2+ \frac{1}{|T|}  \sum_{(x, y) \in T}\EE \left( \dotp{\varphi(x),  \alpha} - \dotp{\varphi(x), \alpha^{\star}}  \right)^2 \notag \\
&= R_{T}^{(\alpha^{\star})}(\alpha^{\star}) + \frac{1}{|T|}  \sum_{(x, y) \in T} \EE\left( \dotp{\varphi(x),  \alpha} - \dotp{\varphi(x), \alpha^{\star}}  \right)^2.
\end{split}
\end{align}
In particular, the risk $R_T$ is minimized at $\alpha^{\star}$. Then, the excess risk of $\widehat \alpha$ is equal to
\begin{align*}
R_{T}^{(\alpha^{\star})}(\widehat \alpha) - R_{T}^{(\alpha^{\star})}(\alpha^{\star}) &= \EE\,\|\widehat \alpha - \alpha^{\star}\|^2_{\Sigma_{T}} \\
&= \EE\,\left\|\left(\Phi(X)^\top \Phi(X)\right)^{-1}\Phi(X)^\top y -\alpha^{\star}\right\|^2_{\Sigma_{T}}\\
&= \EE\,\left\|\left(\Phi(X)^\top \Phi(X)\right)^{-1}\Phi(X)^\top ( y - \Phi(X) \alpha^{\star})\right\|^2_{\Sigma_{T}} \\
&= \EE\,\left\|\left(\Phi(X)^\top \Phi(X)\right)^{-1}\Phi(X)^\top\varepsilon\right\|^2_{\Sigma_{T}} \\
&= \EE\,\Tr\left(\varepsilon^\top\Phi(X)\left(\Phi(X)^\top \Phi(X)\right)^{-1}\Sigma_{T}\left(\Phi(X)^\top \Phi(X)\right)^{-1}\Phi(X)^\top\varepsilon\right)\\
&= \sigma^2 \Tr\left(\left(\Phi(X)^\top \Phi(X)\right)^{-1} \Sigma_{{T}}\right) \\
&= \frac{\sigma^2}{n}\Tr\left(\Sigma_{S}^{-1}\Sigma_{{T}}\right),
\end{align*}
where the second to last inequality uses the cyclic invariance of the trace.

\subsection{Proof of \cref{ex:instantiation-minimax-risk}}\label{app:proof-instantiation-minimax-risk} Let $\bar \Sigma _S = \EE \Sigma_S$ and $ \bar \Sigma_T  = \EE \Sigma_T$. We can upper bound 
\begin{align}\begin{split}\label{eq:break-apart-trace}
    \Tr\left(\Sigma_S^{-1} \Sigma_T\right) & \leq    \underbrace{\Tr\left(\bar \Sigma_S^{-1} \bar \Sigma_T\right)}_{Q_1} +\underbrace{\Tr\left( \left|\Sigma_S^{-1} - \bar \Sigma_S^{-1}\right| \left| \Sigma_T - \bar \Sigma_T\right| \right)}_{Q_2} \\ & \hspace{2cm} + \underbrace{\Tr\left(\left|\Sigma_S^{-1} - \bar \Sigma_S^{-1}\right| \bar \Sigma_T\right)}_{Q_3} + \underbrace{\Tr\left(\bar \Sigma_S^{-1}\left|\Sigma_T - \bar \Sigma_T\right|\right)}_{Q_4}
    \end{split}
\end{align}
where we use $|A|$ to denote the matrix with component-wise absolute values of $A.$ To bound each one of these terms we will use the following lemma. 

\begin{lemma}\label{lem:component-wise-bounds}
    Let $n = |S|$ and fix $\delta \in (0, 1/6)$, then with probability at least $1 - 4\exp(-\delta n)$ we have  
    $$
    \left|\Sigma_S^{-1} - \bar \Sigma_S^{-1}\right|_{11} \leq 2 \delta, \quad \left|\Sigma_S^{-1} - \bar \Sigma_S^{-1}\right|_{12} \leq 2  \sqrt{\frac{\delta}{d}}, \quad \text{and}\quad \left|\Sigma_S^{-1} - \bar \Sigma_S^{-1}\right|_{22} \leq 10 \frac{\delta}{d}.
    $$
    Similarly, let $m = |T|$ and fix $\delta \in (0, \infty),$ with probability at least $ 1- 4\exp(-\delta m)$ we have that 
    $$
      \left|\Sigma_T - \bar \Sigma_T\right|_{11} = 0, \quad \left|\Sigma_T - \bar \Sigma_T\right|_{12} \leq  \sqrt{{\delta}{\bar d}}, \quad \text{and}\quad \left|\Sigma_T - \bar \Sigma_T\right|_{22} \leq 4 {\delta}{\bar d}.
    $$
\end{lemma}
Before proving this lemma, let us show how it proves the conclusion from~\Cref{ex:instantiation-minimax-risk}. Assume that the bounds in the Lemma hold with a fixed $\delta \in (0, 1/6)$. It is immediate that $Q_1 = 1 + \frac{\bar d}{d}.$ Furthermore,
\begin{align*}
    Q_2 &\leq 4 \delta \cdot \frac{\bar d}{d} + 40 \delta^2 \cdot \frac{\bar d} {d} \leq 44 \delta \cdot \left(1 + \frac{\bar d} {d}\right) \\ 
    Q_3 &\leq 2\delta + 10 \delta \frac{\bar d}{d} \leq 10 \delta \cdot \left(1+\frac{\bar d}{d}\right) \\
    Q_4 &\leq 4\delta \cdot \frac{\bar d}{d} \leq 4\delta \cdot \left((1 + \frac{\bar d}{d}\right)
\end{align*}
Combining all of these bounds yields 
$$
 \Tr\left(\Sigma_S^{-1} \Sigma_T\right) \leq (1 + 58\delta)\cdot \left(1 + \frac{\bar d}{d}\right),
$$
setting $\Delta = \delta/58$ proves the promised bound. We now proof Lemma~\ref{lem:component-wise-bounds}.
\begin{proof}[Proof of Lemma~\ref{lem:component-wise-bounds}]
 Let us start with the statement for $S$. Enumerate the samples in $S$ as $(x_1, y_1), \dots, (x_n, y_n).$ For each $x_i$, let $z_i = \sum_{j = 1}^d (x_i)_j$ be the sum of its coordinates. Let $\bar z_1 = \frac{1}{n} \sum_{i=1}^n z_i$ and $\bar z_2 = \frac{1}{n} \sum_{i=1}^n z_i^2.$ Then, we have
 \begin{equation}\label{eq:cov-S}
     \Sigma_S = \begin{pmatrix}
         1 & \bar z_1 \\
          \bar z_1 &  \bar z_2 
     \end{pmatrix}  \quad \text{and} \quad \Sigma_S^{-1} = \frac{1}{\bar z_2 - \bar z_1^2} \begin{pmatrix}
          \bar z_2  & - \bar z_1 \\
          - \bar z_1 &  1
     \end{pmatrix}.
 \end{equation}
 Thus, 
 \begin{equation*}
     \Sigma_S^{-1} - \bar \Sigma_S^{-1} = \begin{pmatrix}
          \frac{\bar z_2}{\bar z_2 - \bar z_1^2} - 1  & - \frac{\bar z_1}{\bar z_2 - \bar z_1^2} \\
          - \frac{\bar z_1}{\bar z_2 - \bar z_1^2}&  \frac{1}{\bar z_2 - \bar z_1^2} - \frac{1}{d}
     \end{pmatrix}.
 \end{equation*}
The random variables $z_i \sim N(0, d)$ and so $\bar z_1 \sim N(0, \frac{d}{n^2})$ and $\bar z_2 =  \frac{d}{n} w $ with $w \sim \chi^2_n$, a chi-squared distribution with $n$ degrees of freedom.
\begin{fact}[\cite{vershynin2018high}]
    Let $z \sim N(0, \sigma^2),$ then for any $t \geq 0$ we have $\PP(|z| \geq t) \leq 2 \exp\left(-\frac{t^2}{2\sigma^2}\right).$
\end{fact}
\begin{fact}[\cite{laurent2000adaptive}]
    Let $w \sim \chi_n^2$, then for any $t \geq 0$ we have that $$\max\left\{\PP( 2\sqrt{n t} + 2t \leq w - n), \PP( w - n \leq - 2\sqrt{n t})\right\} \leq \exp(-t).$$
\end{fact}
Equipped with these two facts, it is easy to derive that for any fixed $\delta > 0,$  we have $\PP(|\bar z_1| \geq  \sqrt{\delta d}) \leq 2 \exp\left(- {\delta n^2}\right)$ and $\max\left\{\PP(2(\delta+\delta^2)d \leq \bar z_2 -d), \PP (\bar z_2 - d \leq -2\delta d) \right\}\leq \exp(-\delta n).$ Consider the event $$
\mathcal{E} = \left\{|\bar z_1 | \leq \sqrt{\delta d} \quad \text{ and } \quad (1 - 2\delta) d \leq \bar z_2 \leq ( 1+4 \delta) d \right\}.
$$
A union bound argument yields $\PP(\cE) \geq 1 - 4\exp\left(-\delta n \right),$ which matches the stated probability. For the rest of the proof assume that $\cE$ holds with $\delta \in (0, 1/6)$.

 We are finally ready to prove the three stated bounds for $S$. We will repeatedly use that $\delta/(1-3\delta) \leq 2 \delta$ for $\delta < 1/6.$ Let us start with the first statement:
    \begin{align*}
       \left| \frac{\bar z_2}{\bar z_2 - \bar z_1^2} -1\right|= \left| \frac{\bar z_1^2}{\bar z_2 - \bar z_1^2}\right| \leq \frac{|\bar z_1^2|}{|\bar z_2| - |\bar z_1^2|} \leq \frac{\delta }{(1-3\delta)} \leq 2 \delta.
    \end{align*}
 Similarly,
    \begin{align*}
       \left| \frac{\bar z_1}{\bar z_2 - \bar z_1^2}\right|= \left| \frac{\bar z_1}{\bar z_2 - \bar z_1^2}\right| \leq \frac{|\bar z_1|}{|\bar z_2| - |\bar z_1^2|} \leq \frac{\sqrt{\delta d}}{(1-3\delta) d} \leq 2\sqrt{\frac{\delta}{d}}
    \end{align*}
    where the last inequality follows for any $\delta < 1/6$ and $d \geq 1.$ Finally, 
\begin{align*}
    \left|\frac{1}{\bar z_2 - \bar z_1^2} - \frac{1}{d} \right| = \left| \frac{d - \bar z_2 + \bar z_1^2}{(\bar z_2 - \bar z_1^2) d}\right| \leq \frac{|d - \bar z_2| + \bar z_1^2}{(|\bar z_2| - |\bar z_1^2|) d } \leq \frac{5 \delta }{(1-3\delta)d} \leq 10 \frac{\delta}{d}.
\end{align*}
This completes the statement involving $S$. 

To show the statement for $T$ we use a very similar strategy. Notice that $\Sigma_T$ can also be written as in \eqref{eq:cov-S} for a $\bar z_1$ and $\bar z_2$ with an analogous distribution where we substitute $d$ with $\bar d.$ Then, assuming again that we are in $\cE$---defined with the the new $\bar z_1$ and $\bar z_2$ and with $\bar d$ in place of $d$---yields the stated probability bound. Hence, $\left(\Sigma_T - \bar \Sigma_T \right)_{11} = 0 $, while 
\begin{align*}
\left|\left(\Sigma_T - \bar \Sigma_T \right)_{12}\right| = |z_1| \leq \sqrt{\delta d}, \quad \text{and} \quad \left|\left(\Sigma_T - \bar \Sigma_T \right)_{22}\right| = |\bar z_2 - d| \leq 4 \delta d,
\end{align*}
which completes the proof.
\end{proof}

\subsection{Proof of Theorem~\ref{thm:loweround}}
    The proof follows along similar lines as the argument in \cite{mourtada2022exact}, also summarized in \cite[Section 3.7]{bach2024learning}.
    Define $n = |S|$. We have already proved an upper bound that matches \eqref{eq:minimax}, so we focus on proving a lower bound. %
    We can lower bound the supremum using a prior on $\alpha \sim N\left(0, \tfrac{\sigma^2}{\lambda n} \cdot I\right)$ where $\lambda >0$ is any fixed positive scalar. Let $Y$ denote the random vector of new observations $Y_i = \dotp{\varphi(x_i), \alpha} + \epsilon_i$ with $x_i\in S$. In particular,
\begin{align}
\begin{split}
    \label{eq:lower-bayes}
        \inf_{\widehat \alpha} \mathop{\sup}_{\alpha}\left\{{R}_T^{(\alpha)}(\widehat \alpha(Y)) - {R}_T^{(\alpha)}( \alpha) \right\} & = \inf_{\widehat \alpha} \mathop{\sup}_{\alpha} \EE_Y \left\| \widehat \alpha - \alpha \right\|_{\Sigma_{T}}^2\\
    & \geq \inf_{\widehat \alpha} \EE_{\alpha} \EE_Y \left[ \left.  \left\| \widehat \alpha(Y) - \alpha \right\|_{\Sigma_{T}}^2 \, \right| \, \alpha \right]\\
   & = \inf_{\widehat \alpha} \EE_{Y} \EE_\alpha \left[\left.  \left\| \widehat \alpha(Y) - \alpha \right\|_{\Sigma_{T}}^2 \right| Y\right] \\
   & \geq \EE_{Y} \inf_{\widehat \alpha} \E_{\alpha} \left[\left.  \left\| \widehat \alpha(Y) - \alpha \right\|_{\Sigma_{T}}^2\right| Y\right],
\end{split}
\end{align}
where the second line uses the prior, and the fourth line uses Jensen's inequality. We can recognize a minimizer for the inner infimum via first-order optimality conditions as the conditional mean
$
\widehat \alpha_{\star} = \EE \left[\alpha \mid Y \right].
$
In order to explicitly compute $\widehat \alpha_{\star}$, observe that
$(\alpha,Y)=(\alpha, \Phi(X)\alpha + \epsilon)$ is a jointly Gaussian vector:
$$ (\alpha, Y)\sim N\left(0, \frac{\sigma^2}{\lambda n}\begin{bmatrix}  I &  \Phi(X)^\top \\
\Phi(X) &  \Phi(X)\Phi(X)^\top +  \lambda n \cdot I\end{bmatrix}\right).$$
Next, we will use the following two standard lemmas, whose proofs are elementary.
\begin{lemma}\label{lem:cond-mean}
    Let $(X,Y) \sim N\left(0, \begin{bmatrix} \Sigma_{xx} & \Sigma{xy} \\ \Sigma_{yx} & \Sigma_{yy}\end{bmatrix}\right),$ then, $\EE\left(X \mid Y\right) = \Sigma_{xy}\Sigma_{yy}^{-1}Y$.
\end{lemma}
\begin{lemma}\label{lem:funny-commutativity}
    Let $Q \in \R^{d \times r}$ be a rank $r$ matrix, then, $Q^\top \left(QQ^\top - I_d\right)^{-1} = \left(Q^\top Q - I_r\right)^{-1} Q^\top.$
\end{lemma}

Thus using Lemma~\ref{lem:cond-mean}, we compute
\begin{align*}
\widehat \alpha_{\star} &=  \frac{\sigma^2}{\lambda n} \Phi(X)^\top \left(\frac{\sigma^2}{\lambda n} \Phi(X) \Phi(X)^\top +  \sigma^2 I\right)^{-1}  Y \\
&= \frac{\sigma^2}{\lambda n} \left( \frac{\sigma^2}{\lambda n}\Phi(X)^\top \Phi(X) +  \sigma^2 I\right)^{-1} \Phi(X)^\top Y \\
&= \left( \Phi(X)^\top \Phi(X) +  \lambda n I\right)^{-1} \Phi(X)^\top Y 
\end{align*}
where the second line follows from Lemma~\ref{lem:funny-commutativity}. Substituting this expression into \eqref{eq:lower-bayes} yields
\begin{align*}
     &\inf_{\widehat \alpha} \mathop{\sup}_{\alpha}\left\{{R}_T^{(\alpha)}(\widehat \alpha) - {R}_T^{(\alpha)}( \alpha) \right\} \\ 
     &\hspace{1cm}\geq \EE_{Y} \EE_{\alpha\mid Y} \left\|  \left( \Phi(X)^\top\Phi(X) +  \lambda n I\right)^{-1} \Phi(X)^\top Y   - \alpha \right\|^2_{\Sigma_{T}} \\
    &\hspace{1cm}= \EE_{\alpha, \varepsilon} \left\|  \left( \Phi(X)^\top\Phi(X) +  \lambda n I\right)^{-1} \Phi(X)^\top\left(  \Phi(X) \alpha + \varepsilon \right) - \alpha \right\|^2_{\Sigma_{T}} \\
        &\hspace{1cm}= \EE_{\alpha, \varepsilon} \left\|  \left( \Phi(X)^\top\Phi(X) +  \lambda n I\right)^{-1} \Phi(X)^\top\left(  \Phi(X) \alpha - \left(\Phi(X) +  \lambda n I\right)\alpha -\varepsilon \right) - \alpha \right\|^2_{\Sigma_{T}} \\
        &\hspace{1cm}= \EE_{\alpha, \varepsilon} \left\|  \left( \Phi(X)^\top\Phi(X) +  \lambda n I\right)^{-1} \Phi(X)^\top\left(\varepsilon - n\lambda \alpha\right)  \right\|^2_{\Sigma_{T}} \\
     & \hspace{1cm}= \EE_{\alpha} \left\|  \lambda\cdot \left(\Sigma_{S} + \lambda\cdot I\right)^{-1} \alpha \right\|^2_{\Sigma_{T}} + \EE_{\varepsilon} \left\|  \left(\Sigma_{S} + \lambda I\right)^{-1} \Phi\left(X\right)^\top \varepsilon  \right\|^2_{\Sigma_{T}} \\
     & \hspace{1cm}= \frac{\lambda \sigma^2}{n} \Tr  \left(\left(\Sigma_{S} + \lambda I\right)^{-2} \Sigma_{T}\right) + \frac{\sigma^2}{n} \Tr\left(\left(\Sigma_S + \lambda I\right)^{-2} \Sigma_S \Sigma_{T}\right) \\
       & \hspace{1cm}= \frac{\sigma^2}{n} \Tr  \left(\left(\Sigma_{S} + \lambda I\right)^{-2} \left(\Sigma_{S}+\lambda I\right)\Sigma_{T}\right) \\
         & \hspace{1cm}=  \frac{\sigma^2}{n} \Tr  \left(\left(\Sigma_S + \lambda I\right)^{-1} \Sigma_{T}\right).
\end{align*}
Taking $\lambda \rightarrow 0$ yields a lower bound of $\frac{\sigma^2}{n} \Tr  \left(\left(\Sigma_S\right)^{-1} \Sigma_{T}\right)$, establishing the result. %

\end{document}